\pdfoutput=1
\documentclass{amsart}
\usepackage{research}
\usepackage{braids}
\usetikzlibrary{decorations.markings,knots}

\DeclareMathOperator{\std}{std}
\DeclareMathOperator{\tb}{tb}
\DeclareMathOperator{\vot}{vot}

\usepackage[style=alphabetic,sorting=nyt]{biblatex}
\appto{\bibsetup}{\raggedright}
\renewbibmacro{in:}{}

\DeclareUnicodeCharacter{0308}{}

\usetikzlibrary{decorations.markings}
\usepackage{multirow}

\tikzset{%
    arrowat/.style={%
        postaction={decorate,decoration={
                markings,
                mark=at position #1 with {\arrow[xshift=2pt]{>}}}}
    } 
}
\tikzset{%
    oparrowat/.style={%
        postaction={decorate,decoration={
                markings,
                mark=at position #1 with {\arrow[xshift=2pt]{<}}}}
    }
}

\RequirePackage{tgpagella}
\RequirePackage[utf8]{inputenc} 
\RequirePackage[T1]{fontenc}
\RequirePackage{amssymb}
\RequirePackage{microtype}

\begin{document}
\title{Some applications of Menke's JSJ decomposition for symplectic fillings}
\author{Austin Christian and Youlin Li}
\begin{abstract}
We apply Menke's JSJ decomposition for symplectic fillings to several families of contact 3-manifolds.  Among other results, we complete the classification up to orientation-preserving diffeomorphism of strong symplectic fillings of lens spaces.  We show that exact symplectic fillings of contact manifolds obtained by surgery on certain Legendrian negative cables are the result of attaching a Weinstein 2-handle to an exact filling of a lens space.  For large families of contact structures on Seifert fibered spaces over $S^2$, we reduce the problem of classifying exact symplectic fillings to the same problem for universally tight or canonical contact structures.  Finally, virtually overtwisted circle bundles over surfaces with genus greater than one and negative twisting number are seen to have unique exact fillings.
\end{abstract}
\maketitle

\section{Introduction and statement of results}\label{sec:intro}
When studying symplectic fillings of contact manifolds, one often wonders whether decompositions which exist for the contact manifold extend to its fillings.  For instance, Eliashberg proved the following result.

\begin{theorem}[{\cite{eliashberg1990filling,cieliebak2012stein}}]\label{thm:connected-sum}
Suppose that a 3-dimensional contact manifold $(M,\xi)$ is obtained from another contact manifold $(M',\xi')$ via connected sum.  Then every symplectic filling of $(M,\xi)$ is obtained by attaching a Weinstein 1-handle to a symplectic filling of $(M',\xi')$.
\end{theorem}

So the symplectic fillings of a contact manifold obtained by connected sum are determined by the fillings of the parties to the connected sum.  Thus, one may attempt to classify the symplectic fillings of a contact manifold $(M,\xi)$ by identifying an embedded sphere along which $(M,\xi)$ decomposes as a connected sum, and then classifying the symplectic fillings of the contact manifolds resulting from this decomposition.  Recently, Menke established a result analogous to that of Eliashberg, decomposing a contact manifold along a torus rather than a sphere; Menke calls this result a \emph{JSJ decomposition} for symplectic fillings, in reference to work of Jaco-Shalen \cite{jaco1978new} and Johannson \cite{johannson1979homotopy}.\\

While Eliashberg's connected sum result allows us to split a contact 3-manifold along any convex sphere, the tori along which Menke's result may be applied are required to satisfy an additional geometric criterion.  A \emph{mixed torus} is an embedded convex torus $T\subset (M,\xi)$ admitting a virtually overtwisted neighborhood of the form $T^2\times[0,2]$, where $T$ is identified with $T^2\times\{1\}$ and each of $T^2\times[0,1]$ and $T^2\times[1,2]$ is a basic slice.  One can then define the notion of \emph{splitting $(M,\xi)$ along $T$} as follows.  Let $s_i$ denote the slope of $T^2\times\{i\}$.  The identification of $T^2$ with $\mathbb{R}^2/\mathbb{Z}^2$ may be normalized so that $s_0=-1$ and $s_1=\infty$.  With this normalization, splitting $(M,\xi)$ with slope $s$ along $T$ will produce a contact manifold $(M',\xi')$.  Here
\[
M' := S_0 \cup_{\psi_0}(M\setminus T)\cup_{\psi_1} S_1,
\]
where each $S_i$ is a solid torus and $\psi_i\colon\partial S_i\to T_i$ is chosen so that the image of a meridian in $\partial S_i$ has slope $s$ in $T_i$.  Notice that the dividing set is vertical, and thus $s$ must be an integer.  We define $\xi'$ to agree with $\xi$ on $M\setminus T$, and on $S_i\subset M'$, $\xi'$ is the unique tight contact structure determined by the characteristic foliation of $\partial S_i$.\\

Finally, where Theorem~\ref{thm:connected-sum} constructs fillings of a contact manifold $(M,\xi)$ by attaching Weinstein 1-handles to fillings of a decomposed contact manifold $(M',\xi')$, the JSJ decomposition for symplectic fillings attaches \emph{round symplectic 1-handles} to fillings.  Round symplectic 1-handle attachment is described in \cite{adachi2017round} and \cite{avdek2021liouville}, and is equivalent to Weinstein 1-handle attachment followed by Weinstein 2-handle attachment.  In particular, attaching a round symplectic 1-handle to a symplectic filling $(W,\omega)$ along Legendrian knots $L_0,L_1$ in its boundary is equivalent to attaching a Weinstein 1-handle to $(W,\omega)$ along points $p_i\in L_i$, $i=0,1$, and then attaching a Weinstein 2-handle to the resulting filling along the knot $L$ obtained by surgering $L_0$ and $L_1$ along $p_0$ and $p_1$.  See \cite[Section 4.2]{avdek2021liouville} or \cite[Section 4]{christian2021symplectic} for further details.\\

At last, we may state Menke's JSJ decomposition for symplectic fillings.

\begin{theorem}[{\cite[Theorem 1.1]{menke2018jsj}}]\label{thm:jsj}
Let $(M,\xi)$ be a closed, cooriented 3-dimensional contact manifold, and let $(W,\omega)$ be an exact symplectic filling of $(M,\xi)$.  If there exists a mixed torus $T^2\subset(M,\xi)$, with normalized embedding $T^2\times[0,2]$, then there exists a (possibly disconnected) symplectic manifold $(W',\omega')$ such that:
\begin{itemize}
	\item $(W',\omega')$ is an exact filling of its boundary $(M',\xi')$;
	\item $(M',\xi')$ is the result of splitting $(M,\xi)$ with some slope $0\leq s\leq s_2-1$ along $T$;
	\item $(W,\omega)$ can be recovered from $(W',\omega')$ by round symplectic 1-handle attachment.
\end{itemize}
\end{theorem}

If $(M',\xi')$ is a contact manifold obtained from $(M,\xi)$ via Legendrian surgery along a Legendrian knot $L\subset(M,\xi)$ which has been stabilized both positively and negatively, a remarkable application of Menke's JSJ decomposition shows that the symplectic fillings of $(M',\xi')$ correspond to those of $(M,\xi)$.

\begin{theorem}[{\cite[Theorem 1.3]{menke2018jsj}}]\label{thm:menke-knot}
Let $L\subset(M,\xi)$ be a Legendrian knot in a contact 3-manifold, and let $(M',\xi')$ be the result of contact surgery on $(M,\xi)$ along $S_+S_-(L)$.  Then every exact symplectic filling of $(M',\xi')$ may be obtained from an exact symplectic filling of $(M,\xi)$ by attaching a Weinstein 2-handle along $S_+S_-(L)$.
\end{theorem}

The purpose of this note is to observe some consequences of Theorems \ref{thm:jsj} and \ref{thm:menke-knot} for the classification of symplectic fillings of virtually overtwisted lens spaces, spaces resulting from surgeries on Legendrian negative cables, certain tight contact structures on Seifert fibered spaces, and virtually overtwisted circle bundles.

\subsection{Lens spaces}
Our first application of Menke's result is to virtually overtwisted lens spaces.  Namely, we prove the following result.

\begin{theorem}\label{thm:lens-space-fillings}
Let $\xi$ be a virtually overtwisted tight contact structure on the lens space $L(p,q)$, with $p>q>0$ and $(p,q)=1$.  Then every strong (respectively, exact) symplectic filling of $(L(p,q),\xi)$ is obtained by attaching a sequence of Weinstein 2-handles to a strong (respectively, exact) symplectic filling of a connected sum of universally tight lens spaces.
\end{theorem}

\begin{remark}
This result has also been obtained by Etnyre-Roy in \cite{etnyre2021symplectic}, where the consequences of this classification are more fully explored.  Moreover, if the universally tight lens spaces which result from Theorem \ref{thm:lens-space-fillings} have their fillings classified up to symplectomorphism, then Etnyre-Roy give a classification of the fillings of the original lens space up to symplectomorphism.
\end{remark}

Note that, while Theorem~\ref{thm:jsj} is stated for exact symplectic fillings, Theorem~\ref{thm:lens-space-fillings} includes statements for both strong and exact symplectic fillings.  For lens spaces, the classification problems for strong and exact symplectic fillings are equivalent.  Per Wendl \cite{wendl2010strongly}, every strong symplectic filling of a contact 3-manifold supported by a planar open book decomposition is symplectic deformation equivalent to a blow-up of a Stein filling of the contact manifold.  But all tight contact structures on lens spaces are supported by planar open book decompositions, according to Schonenberger \cite[Theorem 3.3]{schonenberger2007determining}.  Thus a classification of the exact symplectic fillings of a lens space provides a classification of the strong symplectic fillings, up to symplectic deformation equivalence and blow-up.\\

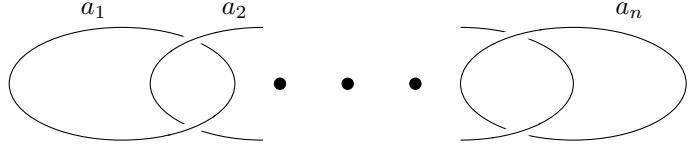
\begin{figure}
\centering
\begin{tikzpicture}[scale=1.5]
\begin{knot}[
	clip width=3.5,
	clip radius=0.31cm,
	ignore endpoint intersections=false,
	flip crossing/.list={2,4}]
\strand (-2,2) arc (90:450:1 and 0.5);
\strand (-0.75,2) arc (90:270:1 and 0.5);
\strand (1,1) arc (-90:90:1 and 0.5);
\strand (3,1.5) arc (0:360:1 and 0.5);
\end{knot}

\begin{scope}[yshift=6cm]
\filldraw[color=black, fill=black](0,-4.5) circle (0.05);
\filldraw[color=black, fill=black](-0.6,-4.5) circle (0.05);
\filldraw[color=black, fill=black](0.6,-4.5) circle (0.05);
\end{scope}

\node[above] at (-2.25,2) {$a_1$};
\node[above] at (-1,2) {$a_2$};
\node[above] at (2.5,2) {$a_n$};
\end{tikzpicture}
\caption{Handlebody diagram for a filling of $L(p,q)$.  We produce a contact structure on $L(p,q)$ by putting each of the unknots in Legendrian position and stabilizing appropriately.}
\label{fig:lens-space-filling}
\end{figure}

By work of Giroux \cite{giroux2000structures} and Honda \cite{honda2000classification}, all tight contact structures on $L(p,q)$ can be described as the contact boundary of a Stein handlebody.  For $p>q>0$, we write
\begin{equation}\label{eq:lens-space-continued-fraction}
-\frac{p}{q} = [a_0,a_1,\ldots,a_n] := a_0 - \cfrac{1}{a_1-\cfrac{1}{\ddots-\cfrac{1}{a_n}}},
\end{equation}
for some uniquely determined integers $a_0,\ldots,a_n\leq -2$.  Then $L(p,q)$ admits $\vert\Pi_{i=1}^n(a_i+1)|$ distinct tight contact structures, up to isotopy.  We realize these contact structures by putting the unknots of Figure \ref{fig:lens-space-filling} into Legendrian position and stabilizing until the framing coefficient becomes $-1$ with respect to the contact framing.  In particular, the knot labeled $a_i$ is stabilized $-2-a_i$ times, giving us $-1-a_i$ choices for how this stabilization is performed.  The universally tight contact structures on $L(p,q)$ are those for which every stabilization (across all knots) is of a single sign.\\

If, in the virtually overtwisted case, our handlebody diagram features a knot $K$ which has been stabilized both positively and negatively, then we may immediately apply Theorem \ref{thm:menke-knot} to conclude that all fillings of $(L(p,q),\xi)$ result from attaching a Weinstein 2-handle to $(L(p',q'),\xi')\#(L(p'',q''),\xi'')$, the connected sum that remains when $K$ is removed from the diagram.  Note that this recovers a result of Plamenevskaya--Van Horn-Morris \cite{plamenevskaya2010planar} which says that $(L(p,1),\xi_{vot})$ has a unique exact filling, for all $p$.  Thus the work of proving Theorem \ref{thm:lens-space-fillings} is reduced to the case where each knot in the handlebody diagram for $(L(p,q),\xi)$ features stabilizations of only one sign, but for which these signs do not all agree.  In such a case we are still able to find a mixed torus, but the contact manifold $\partial(W',\omega')$ which results from applying Theorem \ref{thm:jsj} to a filling $(W,\omega)$ of $(L(p,q),\xi)$ is not uniquely determined.  The possibilities are enumerated in Section \ref{sec:lens-space-proofs}.\\

The classification of symplectic fillings for lens spaces has a long history.  Work of Gromov \cite{gromov1985pseudo} and Eliashberg \cite{eliashberg1990filling} implies that the unique tight contact structures on $S^3$ and $S^1\times S^2$ admit unique exact fillings.  Later, McDuff \cite{mcduff1990structure} showed that the standard tight contact structure on $L(p,1)$ is uniquely fillable, except in the case $p=4$, when there are precisely two exact fillings, up to symplectomorphism.  More generally, Lisca \cite{lisca2008symplectic} obtained a classification up to orientation-preserving diffeomorphism of the symplectic fillings of $(L(p,q),\xi_{\std})$.  In the case of a virtually overtwisted contact structure on $L(p,q)$, we have the above-cited result of Plamenevskaya--Van Horn-Morris, as well as results due to Kaloti \cite{kaloti2013stein}, Fossati \cite{fossati2019contact}, and others for several families of lens spaces.  Theorem \ref{thm:lens-space-fillings} reduces the virtually overtwisted problem to the universally tight problem, and thus completes the classification of strong symplectic fillings of lens spaces up to orientation-preserving diffeomorphism.

\subsection{Surgeries on Legendrian negative cables}
Next we consider spaces obtained from $(S^3,\xi_{\std})$ via contact surgery along certain Legendrian knots.  Theorem \ref{thm:menke-knot} is the first instance of such a result, showing that these surgeries have unique fillings when the Legendrian knot has been stabilized both positively and negatively.  In this section we study fillings in the case that our knot is a Legendrian negative cable of a Legendrian with stabilizations of opposite sign.\\

First defined in \cite{ng2001invariants}, a thorough study of Legendrian satellite knots can be found in \cite{etnyre2018legendrian}, some notation of which we now recall.  We consider a contact manifold $(V,\xi_V)$ defined by $V=D^2_{y,z}\times S^1_\theta$, $\xi_V=\ker(dz-yd\theta)$.  Any Legendrian knot $L\subset(S^3,\xi_{\std})$ has a neighborhood $\nu(L)$ which is contactomorphic to $(V,\xi_V)$, and given any Legendrian knot $Q\subset V$, we denote by $Q(L)\subset\nu(L)$ the image of $Q$ under this contactomorphism.  We pay special attention to the case where $Q\subset V$ is a Legendrian $(p,q)$-torus knot, for some coprime $p,q$ with $q>0$, in which case we call $Q(L)$ a \emph{Legendrian cable} of $L$.  We point out that if $Q$ is a $(p,q)$-torus knot and $\mathcal{K}$ is the knot type of $L$, then the knot type of $Q$ is $\mathcal{K}_{p+q\tb(L),q}$, that of a smooth $(p+q\tb(L),q)$-cable of $L$.  The reason for this is that the contactomorphism between $\nu(L)$ and $V$ identifies the product framing on $V$ with the contact framing on $\nu(L)$; see \cite[Section 5]{etnyre2018legendrian} for more details.\\

We will also need a particular embedding of $(V,\xi_V)$ into itself.  Notice that the core $C$ of $V$ is a Legendrian curve, and that $(V,\xi_V)=\nu(C)$ is a standard neighborhood of $C$.  We may stabilize $C$ to obtain $S_+(C)\subset V$ and identify a standard neighborhood $\nu(S_+(C))\subset V$ of the stabilization.  We have a contactomorphism between $\nu(C)=(V,\xi_V)$ and $\nu(S_+(C))\subset (V,\xi_V)$, giving us an embedding $\zeta\colon(V,\xi_V)\hookrightarrow(V,\xi_V)$.  Given some Legendrian $Q\subset V$, this embedding produces $\zeta(Q)\subset V$, a Legendrian cable of the stabilization $S_+(C)$.\\

Next we point out that a Legendrian knot $Q\subset V$ which is smoothly a $(p,q)$-torus knot can be used to determine a tight contact structure $\xi_Q$ on $L(q^2,pq-1)$.  The construction is as follows: let $S\cong(D^2\times S^1,\xi_{\std})$ be a tight solid torus, glued to $V$ in such a way that $V\cup S\cong(S^2\times S^1,\xi_{\std})$.  Then $(L(q^2,pq-1),\xi_Q)$ is the result of Legendrian surgery on $(S^2\times S^1,\xi_{\std})$ along $\zeta(Q)$.\\

The result of this section will consider fillings of the contact manifold which results from Legendrian surgery on $(S^3,\xi_{\mathrm{std}})$ along a knot $Q(S_+S_-(L))$, under certain conditions on this knot.  In particular, we will show that all such fillings may be obtained by attaching a Weinstein 2-handle to a filling of $(L(q^2,pq-1),\xi_Q)$ along a Legendrian knot $L_Q\subset (L(q^2,pq-1),\xi_Q)$.  Our final preparation before stating the result is to identify the knot $L_Q$.  First, consider the knot $K=\{\mathrm{pt}\}\times S^1$ in $(S^2\times S^1,\xi_{\std})$; we take $K$ to be disjoint from $Q\subset S^2\times S^1$.  By performing the contact connected sum $(S^2\times S^1,\xi_{\std})\#(S^3,\xi_{\std})$ along points $x\in K$ and $y\in S_-(L)$, we obtain $K\# S_-(L)$ as a Legendrian knot in $(S^2\times S^1,\xi_{\std})$.  Finally, we perform Legendrian surgery on $(S^2\times S^1,\xi_{\std})$ along $\zeta(Q)$, and $K\# S_-(L)$ passes to a Legendrian knot in $(L(q^2,pq-1),\xi_Q)$.  This is the Legendrian knot $L_Q$ of interest to us.\\

We are now prepared to state our result.

\begin{theorem}\label{thm:cables}
Let $L\subset(S^3,\xi_{\std})$ be a Legendrian knot with smooth knot type $\mathcal{K}$, and let $Q(S_+S_-(L))$ be a Legendrian negative cable of $S_+S_-(L)$, the smooth knot type of which is $\mathcal{K}_{p,q}$.  Suppose that the Thurston-Bennequin number of $Q(S_+S_-(L))$ is maximal among Legendrian knots of type $\mathcal{K}_{p,q}$, and let $(M,\xi)$ be the contact manifold which results from Legendrian surgery on $(S^3,\xi_{\std})$ along $Q(S_+S_-(L))$.  Then every exact symplectic filling of $(M,\xi)$ may be obtained by attaching a Weinstein 2-handle to an exact symplectic filling of $(L(q^2,pq-1),\xi_Q)$ along $L_Q\subset L(q^2,pq-1)$.
\end{theorem}

\begin{remark}$ $
\begin{enumerate}
	\item Because $Q(S_+S_-(L))$ has the smooth knot type $\mathcal{K}_{p,q}$, this knot is a Legendrian $(p+q(2-\tb(L)),q)$-cable of $S_+S_-(L)$.  Since $Q(S_+S_-(L))$ is a Legendrian negative cable, $p<q(\tb(L)-2)$.
	\item According to \cite[Theorem 5.16]{etnyre2018legendrian}, the Thurston-Bennequin number of $Q(S_+S_-(L))$ is $pq$, and thus $M$ is the result of $(pq-1)$-surgery along $\mathcal{K}_{p,q}$.  By \cite[Corollary 7.3]{gordon1983dehn}, this surgery is diffeomorphic to $(pq-1)/q^2$-surgery along $\mathcal{K}$.
\end{enumerate}
\end{remark}

\subsection{Seifert fibered spaces over $S^2$}\label{subsec:sfs}
In this section we apply Menke's result to large classes of contact structures on spaces which are Seifert fibered over $S^2$, with at least three singular fibers.  Our results reduce the classification of fillings of these spaces to the classification problem for lens spaces --- a problem which is settled by the previous section.  We will first consider Seifert fibered spaces whose Euler number $e_0$ is non-negative, and then consider spaces with $e_0\leq -3$.  Here the \emph{Euler number} of a Seifert fibered space $M(r_1,\ldots,r_n)$ over $S^2$ is defined to be $e_0:=\Sigma\lfloor r_i\rfloor$.  Starkston \cite{starkston2015symplectic} and Choi-Park \cite{choi2019symplectic} have previously studied fillings of small Seifert fibered spaces satisfying $e_0\leq -3$, but we consider a distinct collection of contact structures on these spaces.\\

On small Seifert fibered spaces --- those with precisely three singular fibers --- the contact structures satisfying $e_0\geq 0$ or $e_0\leq -3$ have been classified by Ghiggini-Lisca-Stipsicz \cite{ghiggini2006classification} and Wu \cite{wu2004tight}, and we will see that Menke's results apply to a great many of these structures.  For Seifert fibered spaces over $S^2$ with more than three singular fibers, the tight structures have not been fully classified, but we can construct large classes of tight structures for which Menke's result applies.

\subsubsection{The case $e_0\geq 0$}
We now consider a Seifert fibered space over $S^2$ with $n\geq 3$ singular fibers.  Choose coprime integers $q_i,p_i>0$, $i=1,\ldots,n$, which satisfy $q_i<p_i$ for $i=1,\ldots,n-1$.  We may construct a tight contact structure $\xi$ on $M=M(\frac{q_1}{p_1},\cdots,\frac{q_n}{p_n})$ by realizing $M$ as the boundary of a Stein domain with handlebody description as in Figure \ref{fig:seifert-filling}.  Here
\begin{equation}\label{eq:continued-fraction}
-\frac{p_i}{q_i} = [a_0^i,a_1^i,\ldots,a_{l_i}^i]
\quad
\text{for~}
i=1,\ldots,n,
\end{equation}
for some uniquely determined integers
\[
a_0^n \leq -1
\quad\text{and}\quad
a_0^1,\ldots,a_0^{n-1},a_1^i,\ldots,a_{l_i}^i\leq -2.
\]
We obtain a Stein structure on the handlebody in Figure \ref{fig:seifert-filling} by putting each unknot in Legendrian position with clockwise orientation and stabilizing until the framing coefficient becomes $-1$ with respect to the contact framing.  It is possible for distinct choices of stabilizations to lead to the same tight contact structure on $M$ --- that is, there are equivalence relations among the handlebody diagrams.  For small Seifert fibered spaces, Ghiggini-Lisca-Stipsicz show in \cite{ghiggini2006classification} that there are precisely
\[
\left\vert\left(\prod_{i=1}^3(a_0^i+1)-\prod_{i=1}^3a_0^i\right)\prod_{i=1}^3\prod_{j=1}^{l_i}(a_j^i+1)\right\vert
\]
positive tight contact structures on $M$, up to isotopy.  If $M$ has four singular fibers, Medeto\u{g}ullari shows in \cite{medetogullari2010tight} that the number of distinct Stein fillable contact structures is between
\[
\left\vert\left(\prod_{i=1}^4(a_0^i+1)-\prod_{i=1}^4a_0^i\right)\prod_{i=1}^4\prod_{j=1}^{l_i}(a_j^i+1)\right\vert
\quad\text{and}\quad
2\left\vert\left(\prod_{i=1}^4(a_0^i+1)-\prod_{i=1}^4a_0^i\right)\prod_{i=1}^4\prod_{j=1}^{l_i}(a_j^i+1)\right\vert.
\]
Generally, if $n\geq 4$, then $M$ contains incompressible tori and therefore admits infinitely many tight contact structures according to work of Colin \cite{colin2001torsion,colin2001infinite} and Honda-Kazez-Mati{\'c} \cite{honda2002convex}.\\

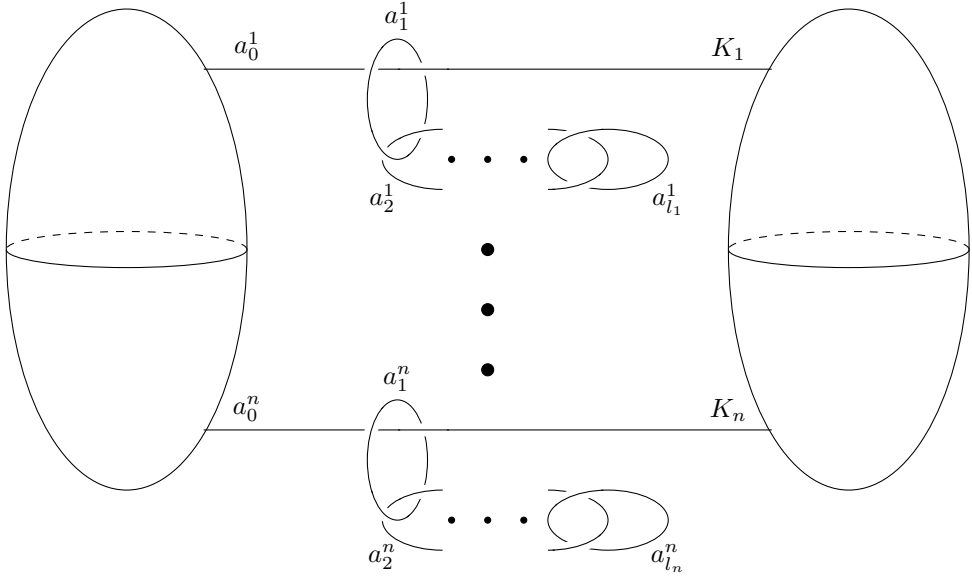
\begin{figure}
\centering
\begin{tikzpicture}[scale=0.8]
\draw (-6,0) ellipse (2 and 4);
\draw (-8,0) arc (180:360:2 and 0.3);
\draw[dashed] (-4,0) arc (0:180:2 and 0.3);
\draw (6,0) ellipse (2 and 4);
\draw (4,0) arc (180:360:2 and 0.3);
\draw[dashed] (8,0) arc (0:180:2 and 0.3);

\begin{knot}[
	clip width=3.5,
	clip radius=0.31cm,
	ignore endpoint intersections=false,
	flip crossing/.list={1,4,6}]
\strand (-4.72,3) to (4.72,3);
\strand (-1.5,3.5) arc (90:450:0.5 and 1);
\strand (-0.75,2) arc (90:270:1 and 0.5);
\strand (1,1) arc (-90:90:1 and 0.5);
\strand (3,1.5) arc (0:360:1 and 0.5);
\end{knot}

\begin{scope}[yshift=-6cm]
\begin{knot}[
	clip width=3.5,
	clip radius=0.31cm,
	ignore endpoint intersections=false,
	flip crossing/.list={1,4,6}]
\strand (-4.72,3) to (4.72,3);
\strand (-1.5,3.5) arc (90:450:0.5 and 1);
\strand (-0.75,2) arc (90:270:1 and 0.5);
\strand (1,1) arc (-90:90:1 and 0.5);
\strand (3,1.5) arc (0:360:1 and 0.5);
\end{knot}
\end{scope}

\filldraw[color=black, fill=black](0,0) circle (0.1);
\filldraw[color=black, fill=black](0,-1) circle (0.1);
\filldraw[color=black, fill=black](0,-2) circle (0.1);
\begin{scope}[yshift=6cm]
\filldraw[color=black, fill=black](0,-4.5) circle (0.05);
\filldraw[color=black, fill=black](-0.6,-4.5) circle (0.05);
\filldraw[color=black, fill=black](0.6,-4.5) circle (0.05);
\end{scope}
\filldraw[color=black, fill=black](0,-4.5) circle (0.05);
\filldraw[color=black, fill=black](-0.6,-4.5) circle (0.05);
\filldraw[color=black, fill=black](0.6,-4.5) circle (0.05);

\node[above] at (4,3) {$K_1$};
\node[above] at (-4,3) {$a_0^1$};
\node[above] at (-1.5,3.5) {$a_1^1$};
\node[below] at (-1.75,1.25) {$a_2^1$};
\node[below] at (3,1.25) {$a_{l_1}^1$};

\begin{scope}[yshift=-6cm]
\node[above] at (4,3) {$K_n$};
\node[above] at (-4,3) {$a_0^n$};
\node[above] at (-1.5,3.5) {$a_1^n$};
\node[below] at (-1.75,1.25) {$a_2^n$};
\node[below] at (3,1.25) {$a_{l_n}^n$};
\end{scope}
\end{tikzpicture}
\caption{Handlebody decomposition of a Stein filling of $M=M(\frac{q_1}{p_1},\cdots,\frac{q_n}{p_n})$.  A contact structure is produced on $M$ by putting each of the knots in Legendrian position and stabilizing appropriately. }
\label{fig:seifert-filling}
\end{figure}

Our first result for spaces $M$ which are Seifert fibered over $S^2$ applies to tight contact structures which are \emph{thoroughly mixed}, a notion we define shortly.  The definition is designed so that each singular fiber of $M$ admits a tubular neighborhood whose boundary is a mixed torus.  Applying Theorem \ref{thm:jsj} to a filling of $(M,\xi)$ will then leave us with a boundary connected sum of fillings of lens spaces.

\begin{definition}
Let $\xi$ be a tight contact structure on $M(\frac{q_1}{p_1},\cdots,\frac{q_n}{p_n})$.  We will call $\xi$ \emph{thoroughly mixed} if $\xi$ admits a Stein filling as in Figure~\ref{fig:seifert-filling} such that one of the following holds:
\begin{itemize}
	\item if $e_0=0$, then each of $K_1,\ldots,K_n$ has been stabilized positively (or, equivalently, each has been stabilized negatively);
	\item if $e_0>0$, then each of $K_1,\ldots,K_{n-1}$ has been stabilized positively, and the nearest stabilized unknot adjacent to $K_n$ has also been stabilized positively (or, equivalently, each of these stabilizations is negative.
\end{itemize}
Notice that $K_n$ has no stabilizations in the case $e_0>0$.
\end{definition}

\begingroup
\tikzset{every picture/.style={scale=0.56}}
\begin{figure}
\centering
\begin{subfigure}{0.45\linewidth}
\centering
\begin{tikzpicture}[xscale=0.8,yscale=0.7]
\draw (-6,0) ellipse (2 and 4);
\draw (-8,0) arc (180:360:2 and 0.3);
\draw[dashed] (-4,0) arc (0:180:2 and 0.3);
\draw (6,0) ellipse (2 and 4);
\draw (4,0) arc (180:360:2 and 0.3);
\draw[dashed] (8,0) arc (0:180:2 and 0.3);

\draw (-4.72,3) to[out=0,in=180]
	  (-2,3) to[out=180,in=0]
	  (-3,2.5) to[out=0,in=180]
	  (3,2.5) to[out=180,in=0]
	  (2,3) to[out=0,in=180]
	  (4.72,3);
\draw (-4,0) to[out=0,in=180]
	  (-2,0) to[out=180,in=0]
	  (-3,-0.5) to[out=0,in=180]
	  (4,0);

\begin{scope}[yshift=-6cm]
\draw (-4.72,3) to[out=0,in=180]
	  (-2,3) to[out=180,in=0]
	  (-3,2.5) to[out=0,in=180]
	  (4.72,3);
\end{scope}

\node[above] at (4,3) {$K_1$};
\node[above] at (3.25,0) {$K_2$};
\begin{scope}[yshift=-6cm]
\node[above] at (3.75,3) {$K_3$};
\end{scope}
\end{tikzpicture}
\caption{A thoroughly mixed contact structure.}
\label{fig:thoroughly-mixed-surgery-diagram}
\end{subfigure}~~~~~
\begin{subfigure}{0.45\linewidth}
\centering
\begin{tikzpicture}[xscale=0.8,yscale=0.7]
\draw (-6,0) ellipse (2 and 4);
\draw (-8,0) arc (180:360:2 and 0.3);
\draw[dashed] (-4,0) arc (0:180:2 and 0.3);
\draw (6,0) ellipse (2 and 4);
\draw (4,0) arc (180:360:2 and 0.3);
\draw[dashed] (8,0) arc (0:180:2 and 0.3);

\draw (-4.72,3) to[out=0,in=180]
	  (-2,3) to[out=180,in=0]
	  (-3,2.5) to[out=0,in=180]
	  (3,2.5) to[out=180,in=0]
	  (2,3) to[out=0,in=180]
	  (4.72,3);
\draw (-4,0) to[out=0,in=180]
	  (-2,0) to[out=180,in=0]
	  (-3,-0.5) to[out=0,in=180]
	  (4,0);

\begin{scope}[yshift=-6cm]
\draw (-4.72,3) to[out=0,in=180]
	  (3,2.5) to[out=180,in=0]
	  (2,3) to[out=0,in=180]
	  (4.72,3);
\end{scope}

\node[above] at (4,3) {$K_1$};
\node[above] at (3.25,0) {$K_2$};
\begin{scope}[yshift=-6cm]
\node[above] at (3.75,3) {$K_3$};
\end{scope}
\end{tikzpicture}
\caption{A lightly mixed contact structure.}
\label{fig:lightly-mixed-surgery-diagram}
\end{subfigure}
\caption{Contact structures on $M(\frac{1}{3},\frac{1}{2},\frac{1}{2})$.}
\label{fig:mixed-surgery-diagrams}
\end{figure}
\endgroup

See Figure~\ref{fig:mixed-surgery-diagrams} for an example of a thoroughly mixed contact structure.  In Section~\ref{subsubsec:mixed-structures} we will provide another construction of thoroughly mixed tight contact structures and explain how to produce a mixed torus for each singular fiber.

\begin{theorem}\label{thm:thoroughly-mixed}
Let $\xi$ be a tight contact structure on the Seifert fibered space $M=M(\frac{q_1}{p_1},\cdots,\frac{q_n}{p_n})$, for some $n\geq 3$ and coprime positive integers $q_i,p_i$ with $q_i<p_i$ for $1\leq i\leq n-1$ and $p_i\geq 2$ for $1\leq i \leq n$.  If $\xi$ is thoroughly mixed, then there are tight contact structures $\xi_i$ on $L(q_i,-p_i)$ for $i=1,\ldots,n$ and Legendrian knots $L_i^-\subset(L(q_i,-p_i),\xi_i)$, $L_i^+\subset(L(q_{i+1},-p_{i+1}),\xi_{i+1})$ for $i=1,\ldots,n-1$ such that every exact symplectic filling of $(M,\xi)$ is obtained from a disjoint union of exact fillings of $(L(q_1,-p_1),\xi_1),\ldots,(L(q_n,-p_n),\xi_n)$ by attaching a round symplectic 1-handle along $L_i^\pm$, for $i=1,\ldots,n-1$.
\end{theorem}

Several families of tight lens spaces are known to have unique exact fillings, and from these we obtain families of tight Seifert fibered spaces with unique exact fillings.

\begin{corollary}\label{cor:thorough}
Let $\xi$ be a thoroughly mixed tight contact structure on $M=M\left(\frac{q_1}{p_1},\cdots,\frac{q_n}{p_n}\right)$, with $q_i<p_i$ for $1\leq i\leq n-1$ and $p_i\geq 2$ and $\gcd(q_i,p_i)=1$ for $1\leq i \leq n$.  If any of the following conditions hold, then $(M,\xi)$ admits a unique exact symplectic filling, up to symplectomorphism:
\begin{enumerate}[label=(\alph*)]
	\item $q_i\in\{1,2,3\}$; \label{cond:small-q}
	\item for some $b_0^{(i)}-2>b_1^{(i)}\geq 2$ and $m_1,\ldots,m_{n-1}\geq 2$, $m_n\geq 1$, we have
	\[
	\dfrac{q_i}{p_i} = \dfrac{b_0^{(i)}b_1^{(i)}+1}{m_i(b_0^{(i)}b_1^{(i)}+1)-b_1^{(i)}}
	\]
	for $i=1,\ldots,n$; \label{cond:kaloti}
	\item for some $b_0^{(i)},b_1^{(i)}\geq 5$ and $m_1,\ldots,m_{n-1}\geq 2$, $m_n\geq 1$, we have
	\[
	\dfrac{q_i}{p_i} = \dfrac{b_0^{(i)}b_1^{(i)}-1}{m_i(b_0^{(i)}b_1^{(i)}-1)-b_1^{(i)}}
	\]
	for $i=1,\ldots,n$. \label{cond:fossati}
\end{enumerate}
\end{corollary}
\begin{proof}
Theorem \ref{thm:thoroughly-mixed} provides a recipe for constructing any exact filling of $(M,\xi)$ from fillings of $(L(q_i,-p_i),\xi_i)$, so this is simply a matter of observing that if any of these conditions hold, then $(L(q_i,-p_i),\xi_i)$ is uniquely fillable.  If condition \ref{cond:small-q} holds, then either $L(q_i,-p_i)=L(q_i,1)$ or $L(q_i,-p_i)=L(3,2)$.  In either case, the tight contact structures on $L(q_i,-p_i)$ are all uniquely fillable by work of Eliashberg \cite{eliashberg1990filling}, McDuff \cite{mcduff1990structure}, and Plamenevskaya--Van Horn-Morris \cite{plamenevskaya2010planar}.  When condition \ref{cond:kaloti} holds, we are considering tight contact structures on $L(b_0^{(i)}b_1^{(i)}+1,b_1^{(i)})$.  Universally tight structures on such a lens space were shown to be uniquely fillable by Lisca \cite{lisca2008symplectic}.  In particular, we have $p=b_0^{(i)}b_1^{(i)}+1$ and $q=b_1^{(i)}$, so
\[
\dfrac{p}{p-q} = \dfrac{b_0^{(i)}b_1^{(i)}+1}{b_0^{(i)}b_1^{(i)}+1-b_1^{(i)}} = [2,\ldots,2,b_{1}^{(i)}+1],
\]
where the number of copies of $2$ at the start of this continued fraction is $b_0^{(i)}$.  In Lisca's notation, it follows that the unique exact symplectic filling of $L(p,q)$ is $W_{p,q}((1,2,\ldots,2,1))$.  The virtually overtwisted structures on $L(b_0^{(i)}b_1^{(i)}+1,b_1^{(i)})$ are uniquely fillable according to work of Kaloti \cite[Theorem 1.10]{kaloti2013stein}.  Similarly, condition \ref{cond:fossati} produces lens spaces of the form $L(b_0^{(i)}b_1^{(i)}-1,b_1^{(i)})$, the fillings of which are known to be unique by work of Lisca \cite{lisca2008symplectic} in the universally tight case and Fossati \cite[Theorem 1]{fossati2019contact} in the virtually overtwisted case.  The relevant continued fraction for applying Lisca's work to these lens spaces is
\[
\dfrac{b_0^{(i)}b_1^{(i)}-1}{b_0^{(i)}b_1^{(i)}-1-b_1^{(i)}} = [2,\ldots,2,3,2,\ldots,2],
\]
which begins with $b_0^{(i)}-2$ copies of 2 and ends with $b_1^{(i)}-2$ copies.  Once again, the unique exact filling is given by $W_{p,q}((1,2,\ldots,2,1))$ in Lisca's notation.  The observation that
\[
\dfrac{b_0^{(i)}b_1^{(i)}-1}{b_1^{(i)}} = [b_0^{(i)},b_1^{(i)}]
\]
allows us to apply Fossati's result.
\end{proof}

Another class of contact structures on Seifert fibered spaces which admit abundant mixed tori are those which are \emph{lightly mixed}.

\begin{definition}
Let $\xi$ be a tight contact structure on $M(\frac{q_1}{p_1},\cdots,\frac{q_n}{p_n})$.  We will call $\xi$ \emph{lightly mixed} if $\xi$ is not thoroughly mixed, but admits a Stein filling as in Figure~\ref{fig:seifert-filling} for which at least $n-2$ of $K_1,\ldots,K_n$ have been stabilized both positively and negatively.  We say that $\xi$ is \emph{lightly mixed about $K_i$ and $K_j$} to indicate that $\xi$ admits a Stein filling for which each of $K_1,\ldots,K_n$ except $K_i$ and $K_j$ have been stabilized positively and negatively.
\end{definition}

Like their thoroughly mixed counterparts, exact symplectic fillings of lightly mixed contact structures on Seifert fibered spaces may also be decomposed into lens space fillings, though one of the lens spaces will have a slightly more complicated expression.

\begin{theorem}\label{thm:lightly-mixed}
Let $\xi$ be a tight contact structure on the Seifert fibered space $M=M(\frac{q_1}{p_1},\cdots,\frac{q_n}{p_n})$, for some coprime positive integers $q_i,p_i$ with $q_i<p_i$ for $1\leq i \leq n-1$ and $p_i\geq 2$ for $1\leq i \leq n$.  Let each $-p_i/q_i$ have continued fraction as above.  Suppose that $\xi$ is lightly mixed about $K_i$ and $K_j$, and let
\[
-\frac{p'}{q'} = [a_{l_{i}}^{i},\ldots,a_1^{i},a_0^{i}+a_0^{j},a_1^{j},\ldots,a_{l_{j}}^{j}].
\]
Then there exist
\begin{enumerate*}[label=(\arabic*)]
	\item a tight contact structure $\xi_k$ on $L(q_k,-p_k)$, for each $k\neq i,j$;
	\item a tight contact structure $\zeta'$ on $L(p',q')$;
	\item Legendrian knots $K'_k$ in $\mathop{\#}_{k\neq i,j} (L(q_k,-p_k),\xi_k) \# (L(p',q'),\zeta')$ for $k\neq i,j$,
\end{enumerate*}
such that every exact symplectic filling of $(M,\xi)$ is obtained from an exact symplectic filling of
\[
\mathop{\#}\limits_{k\neq i,j} (L(q_k,-p_k),\xi_k) \# (L(p',q'),\zeta')
\]
by attaching a Weinstein 2-handle along each $K'_k$, $k\neq i,j$.
\end{theorem}

As with Theorem \ref{thm:thoroughly-mixed}, we may use Theorem \ref{thm:lightly-mixed} and the classification of exact fillings of some lens spaces to identify families of tight Seifert fibered spaces whose exact fillings we may classify.  One example of such a family is given by the following corollary.

\begin{corollary}\label{cor:lightly}
Choose $p_1,p_2,p_3\geq 2$.  If $\xi$ is a tight contact structure on $M=M(\frac{1}{p_1},\frac{1}{p_2},\frac{1}{p_3})$ which is lightly mixed about $K_{i-1}$ and $K_{i+1}$ --- where subscripts are labeled modulo 3 --- and $p_{i-1}+p_{i+1}\neq 4$, then $(M,\xi)$ admits a unique exact filling, up to symplectomorphism.
\end{corollary}
\begin{proof}
Applying Theorem \ref{thm:lightly-mixed} to such a filling leaves us with a filling of $S^3\# L(p_{i-1}+p_{i+1},1)$, with some tight contact structure.  By work of Eliashberg \cite{eliashberg1990filling} this filling must be the boundary connected sum of a filling of $(S^3,\xi_{\std})$ with a filling of $(L(p_{i-1}+p_{i+1},1),\zeta)$.  Because $p_{i-1}+p_{i+1}$ is not equal to 4, results of McDuff \cite{mcduff1990structure} and Plamenevskaya--Van Horn-Morris \cite{plamenevskaya2010planar} show that $(L(p_{i-1}+p_{i+1},1),\zeta)$ is uniquely fillable, as is the case for $(S^3,\xi_{\std})$.  So $(M,\xi)$ is uniquely fillable.
\end{proof}

We will see in Section~\ref{subsubsec:mixed-structures} that there are precisely six tight contact structures on $M(\frac{1}{p_1},\frac{1}{p_2},\frac{1}{p_3})$ which are neither lightly nor thoroughly mixed, and we will show that each of these structures is universally tight.  According to Corollaries \ref{cor:thorough} and \ref{cor:lightly}, these six are the only tight structures on $M(\frac{1}{p_1},\frac{1}{p_2},\frac{1}{p_3})$ which we cannot conclude have unique exact fillings if, say, $p_1,p_2,p_3\geq 3$.

\begin{corollary}\label{corollary:vot}
Choose integers $p_1,p_2,p_3\geq 2$, no two of which sum to 4.  If $\xi$ is a virtually overtwisted tight contact structure on $M=M(\frac{1}{p_1},\frac{1}{p_2},\frac{1}{p_3})$, then $(M,\xi)$ admits a unique exact filling $(W,\omega)$, up to symplectomorphism.  Moreover, $W$ is simply connected, and has $H_2(W)=\mathbb{Z}^2$.
\end{corollary}
\begin{proof}
The uniqueness of the filling $(W,\omega)$ follows from Corollaries \ref{cor:thorough} and \ref{cor:lightly}.  To see that $W$ is simply connected and has $H_2(W)=\mathbb{Z}^2$, consider the handlebody diagram for $W$ given by Figure 2.  This diagram consists of a single 1-handle, with three 2-handles attached along parallel knots $K_1,K_2,K_3$ which pass over the 1-handle.  We may handleslide $K_2$ and $K_3$ over $K_1$ and then cancel $K_1$ with the 1-handle to obtain a handlebody diagram for $W$ which consists of two 2-handles attached to a 0-handle.  Such a handlebody is simply connected, with $H_2(W)=\mathbb{Z}^2$.
\end{proof}

Wrapping up loose ends, we have the following result.  In this result, we refer to a \emph{horizontal link} in Figure~\ref{fig:seifert-filling}, which consists of some $K_i$ along with the attached chain of Legendrian unknots.

\begin{theorem}\label{thm:not-mixed}
Let $\xi$ be a tight contact structure on a small Seifert fibered space $M$, with surgery diagram as in Figure \ref{fig:seifert-filling}.  If any of the horizontal links have both positive and negative stabilizations, then every exact symplectic filling of $(M,\xi)$ can be obtained from a disjoint union of a filling of a universally tight small Seifert fibered space, along with fillings of universally tight lens spaces, by attaching a sequence of round symplectic 1-handles.
\end{theorem}

\subsubsection{The case $e_0\leq -3$}
Finally, we discuss Seifert fibered spaces over $S^2$ with Euler number $e_0\leq -3$.  In particular, we consider $M=M(-\frac{q_1}{p_1},\cdots,-\frac{q_n}{p_n})$ with $p_i\geq 2$, $q_i\geq 1$, and $(p_i,q_i)=1$ for $i=1,\ldots,n$.  The Euler number is then given by
\[
e_0 = \sum_{i=1}^n \left\lfloor-\frac{q_i}{p_i}\right\rfloor\leq -n.
\]
We have continued fraction expansions
\[
-\frac{q_i}{p_i} = [a_0^i,\ldots,a_{l_i}^i],
\]
for some uniquely determined integers satisfying $a_0^i=-(\lfloor\frac{q_i}{p_i}\rfloor+1)$ and $a_j^i\leq -2$ for $j\geq 1$.  Then $M$ admits a surgery diagram as in Figure \ref{fig:negative-euler}.  Notice that $e_0=\Sigma_{i=1}^n a_0^i$.\\

\begin{figure}
\centering
\begin{tikzpicture}[scale=0.8]
\begin{knot}[
	clip width=5,
	clip radius=2pt,
	ignore endpoint intersections=false,
	flip crossing/.list={2,3,6,7,10,11}]
\strand (0,0) circle (1.75);
\strand (-1.95,-.71) circle (1);
\strand (1.95,-.71) circle (1);
\strand (-3.15,-2.42) arc (-90:145:1);
\strand (3.15,-2.42) arc (270:35:1);
\strand (-5.55,-1.84) arc (90:325:1);
\strand (5.55,-1.84) arc (90:-140:1);
\strand (-6.75,-3.55) circle (1);
\strand (6.75,-3.55) circle (1);
\end{knot}

\filldraw[color=black, fill=black](0,-4) circle (0.1);
\filldraw[color=black, fill=black](-1.25,-3.65) circle (0.1);
\filldraw[color=black, fill=black](1.25,-3.65) circle (0.1);

\filldraw[color=black, fill=black](-4.35,-2.13) circle (0.05);
\filldraw[color=black, fill=black](-4.95,-2.48) circle (0.05);
\filldraw[color=black, fill=black](-3.75,-1.77) circle (0.05);

\filldraw[color=black, fill=black](4.35,-2.13) circle (0.05);
\filldraw[color=black, fill=black](4.95,-2.48) circle (0.05);
\filldraw[color=black, fill=black](3.75,-1.77) circle (0.05);

\node[above] at (0,1.75) {$e_0$};
\node[above] at (-2,0.29) {$a_1^1$};
\node[above] at (2,0.29) {$a_1^n$};
\node[above] at (-3.2,-0.42) {$a_2^1$};
\node[above] at (3.2,-0.42) {$a_2^n$};
\node[above] at (-5.6,-1.84) {$a_{\ell_1-1}^1$};
\node[above] at (5.6,-1.84) {$a_{\ell_n-1}^n$};
\node[above] at (-6.8,-2.55) {$a_{\ell_1}^1$};
\node[above] at (6.8,-2.55) {$a_{\ell_n}^n$};

\node[below right] at (-2,-1.71) {$K_1^1$};
\node[below left] at (2,-1.71) {$K_1^n$};
\node[below right] at (-3.2,-2.42) {$K_2^1$};
\node[below left] at (3.2,-2.42) {$K_2^n$};
\node[below right] at (-5.6,-3.84) {$K_{\ell_1-1}^1$};
\node[below left] at (5.6,-3.84) {$K_{\ell_n-1}^n$};
\node[below right] at (-6.8,-4.55) {$K_{\ell_1}^1$};
\node[below left] at (6.8,-4.55) {$K_{\ell_n}^n$};
\end{tikzpicture}
\caption{A surgery diagram for $M(-\frac{q_1}{p_1},\cdots,-\frac{q_n}{p_n})$, $e_0\leq -3$.}
\label{fig:negative-euler}
\end{figure}

We may construct contact structures on $M$ by putting the knots in Figure \ref{fig:negative-euler} into Legendrian position and stabilizing until the framing coefficient becomes $-1$ with respect to the contact framing.  We see that there are
\[
\left\vert(e_0+1)\prod_{i=1}^n\prod_{j=1}^{\ell_i}(a_j^i+1)\right\vert
\]
choices for these stabilizations, and in case we have a small Seifert fibered space, Wu shows in \cite{wu2004tight} that each such choice leads to a distinct contact structure up to isotopy, and indeed all contact structures on $M$ can be constructed in this way.

\begin{definition}
Let $M=M(-\frac{q_1}{p_1},\cdots,-\frac{q_n}{p_n})$ be as above, and construct a tight contact structure $\xi$ on $M$ by putting the knots of Figure \ref{fig:negative-euler} into Legendrian position.  We say that $(M,\xi)$ is \emph{centrally mixed} if the central knot of Figure \ref{fig:negative-euler} is stabilized both positively and negatively.
\end{definition}

\begin{remark}
Notice that if $(M,\xi)$ is centrally mixed, then $e_0\leq -4$.
\end{remark}

\begin{figure}
\centering
\begin{tikzpicture}[scale=1]
\filldraw[color=black, fill=black](0,0) circle (0.1);

\draw (0,0) -- (1,1);
\draw (0,0) -- (1,0);
\draw (0,0) -- (1,-1);

\filldraw[color=black, fill=black](1,1) circle (0.1);
\filldraw[color=black, fill=black](2,1) circle (0.1);\filldraw[color=black, fill=black](3.25,1) circle (0.05);
\filldraw[color=black, fill=black](3.5,1) circle (0.05);
\filldraw[color=black, fill=black](3.75,1) circle (0.05);
\filldraw[color=black, fill=black](5,1) circle (0.1);

\draw (1,1) -- (3,1);
\draw (4,1) -- (5,1);

\filldraw[color=black, fill=black](1.5,0.3) circle (0.05);
\filldraw[color=black, fill=black](1.5,0) circle (0.05);
\filldraw[color=black, fill=black](1.5,-0.3) circle (0.05);

\filldraw[color=black, fill=black](1,-1) circle (0.1);
\filldraw[color=black, fill=black](2,-1) circle (0.1);\filldraw[color=black, fill=black](3.25,-1) circle (0.05);
\filldraw[color=black, fill=black](3.5,-1) circle (0.05);
\filldraw[color=black, fill=black](3.75,-1) circle (0.05);
\filldraw[color=black, fill=black](5,-1) circle (0.1);

\draw (1,-1) -- (3,-1);
\draw (4,-1) -- (5,-1);

\filldraw[color=black, fill=black](4.5,0.3) circle (0.05);
\filldraw[color=black, fill=black](4.5,0) circle (0.05);
\filldraw[color=black, fill=black](4.5,-0.3) circle (0.05);

\node[left] at (0,0) {$e_0$};

\node[below] at (1,-1) {$a_1^1$};
\node[below] at (2,-1) {$a_2^1$};
\node[below] at (5,-1) {$a_{\ell_1}^1$};

\node[above] at (1,1) {$a_1^n$};
\node[above] at (2,1) {$a_2^n$};
\node[above] at (5,1) {$a_{\ell_n}^n$};
\end{tikzpicture}
\caption{The plumbing graph associated to the surgery diagram in Figure \ref{fig:negative-euler}.}
\label{fig:plumbing-graph}
\end{figure}

Our Seifert fibered space $M$ admits a canonical contact structure as the boundary of a plumbing 4-manifold.  The 4-manifold is a plumbing of disc bundles of 2-spheres, with plumbing graph as in Figure \ref{fig:plumbing-graph}.  Each node of the graph corresponds to a symplectic 2-sphere, and each edge represents an orthogonal intersection between them; in this way we produce a symplectic structure on the plumbing 4-manifold, and a canonical contact structure on its boundary $M$.  The fillings of this canonical contact structure were studied by Starkston \cite{starkston2015symplectic} and Choi-Park \cite{choi2019symplectic}, each under some additional assumptions on $M$.  Starkston provided topological restrictions on the strong symplectic fillings of \emph{dually positive} Seifert fibered spaces over $S^2$.  In some cases, these restrictions produce classifications up to diffeomorphism of minimal strong symplectic fillings.  Choi-Park classified all minimal symplectic fillings of small Seifert 3-manifolds $M(-\frac{q_1}{p_1},-\frac{q_2}{p_2},-\frac{q_3}{p_3})$, with $e_0\leq -4$ and with the canonical contact structure.  An infinite family of Seifert fibered spaces with canonical contact structure and $e_0=-3$ also saw their fillings classified by Sch\"{o}nenberger in \cite[Theorem 4.4]{schonenberger2007determining}.\\

We point out that the canonical contact structure is not centrally mixed.  Indeed, the Legendrian surgery diagram for the canonical contact structure has all of its stabilizations of a single sign.

\begin{proposition}\label{prop:canonical-contact-structure}
If $M=M(-\frac{q_1}{p_1},\cdots,-\frac{q_n}{p_n})$ has canonical contact structure $\xi$ as described above, and Legendrian surgery diagram as in Figure \ref{fig:negative-euler}, then all the stabilizations in the Legendrian surgery diagram are of a single sign.
\end{proposition}
\begin{proof}
As described above, there is a symplectic 4-manifold $(W,\omega)$, obtained by plumbing disc bundles of 2-spheres, which fills $(M,\xi)$.  For each symplectic sphere $S_i^j$, the adjunction formula takes the form
\[
\langle c_1(W),[S_i^j]\rangle = [S_i^j]\cdot[S_i^j]+2.
\]
At the same time, $S_i^j$ corresponds to surgery along the Legendrian knot $K_i^j$ in Figure \ref{fig:negative-euler}, and thus
\[
[S_i^j]\cdot[S_i^j] = \mathrm{fr}(K_i^j) = \mathrm{tb}(K_i^j)-1.
\]
So $\langle c_1(W),[S_i^j]\rangle=\mathrm{tb}(K_i^j)+1$.  Finally, \cite[Proposition 2.3]{gompf1998handlebody} allows us to compute the rotation number of $K_i^j$:
\[
\mathrm{rot}(K_i^j) = \langle c_1(W),[S_i^j]\rangle = \mathrm{tb}(K_i^j)+1.
\]
This rotation number can only be obtained by taking every stabilization to be negative.  Note that taking all stabilizations to be positive gives a contactomorphic (though not isotopic) contact structure.
\end{proof}

A straightforward consequence of the definition of centrally mixed and Theorem \ref{thm:menke-knot} is the following.

\begin{proposition}\label{prop:centrally-mixed}
Let $M=M(-\frac{q_1}{p_1},\ldots,-\frac{q_n}{p_n})$, with $p_i\geq 2$, $q_i\geq 1$, with $n\geq 3$.  If $\xi$ is a centrally mixed tight contact structure on $M$, then every exact symplectic filling of $(M,\xi)$ may be obtained from an exact symplectic filling of
\[
\mathop{\#}\limits_{i=1}^n (L(p_i',q_i'),\xi_i)
\]
by attaching a Weinstein 2-handle in a specified manner, where $-\frac{p_i'}{q_i'} = [a_1^i,\ldots,a_{l_i}^i]$ and $\xi_i$ is a tight contact structure determined by $\xi$.
\end{proposition}

More generally, we have the following result for any contact structures constructed from Figure \ref{fig:negative-euler}.

\begin{theorem}\label{thm:negative-euler}
Let $M=M(-\frac{q_1}{p_1},\cdots,-\frac{q_n}{p_n})$ be as above, and let $\xi$ be a tight contact structure on $M$ obtained by putting the knots of Figure \ref{fig:negative-euler} into Legendrian position.  Then every exact symplectic filling of $(M,\xi)$ can be obtained by attaching a sequence of round symplectic 1-handles to a disjoint union of fillings of universally tight lens spaces and a Seifert fibered space with canonical contact structure.
\end{theorem}

This result will be made more precise through a sequence of propositions in Section \ref{sec:negative-euler}.  We point out that, together with Lisca's classification of fillings for universally tight lens spaces and Choi-Park's classification of fillings for the canonical contact structure on $M(-\frac{q_1}{p_1},-\frac{q_2}{p_2},-\frac{q_3}{p_3})$, $e_0\leq -4$, Theorem \ref{thm:negative-euler} allows us to classify the exact symplectic fillings of any contact structure on $M(-\frac{q_1}{p_1},-\frac{q_2}{p_2},-\frac{q_3}{p_3})$, $e_0\leq -4$.

\subsection{Virtually overtwisted circle bundles over surfaces}\label{subsec:circle-bundles}
Our final application of Menke's JSJ decomposition completes the classification of exact symplectic fillings for virtually overtwisted tight contact structures on circle bundles over closed surfaces.  We let $\pi\colon M\to\Sigma$ be a circle bundle over a closed Riemann surface $\Sigma$ of genus $g$, and we let $\xi$ be a tight contact structure on $M$.  Honda \cite{honda2000classification2} defines the \emph{twisting number} $t(S^1)\leq 0$ of $\xi$ to be the maximum non-positive twisting number achieved by a closed Legendrian curve in $M$ which is isotopic to the $S^1$-fiber.  Here the twisting number is measured relative to the fibration framing, and is defined to be zero if $M$ admits a fiber-isotopic Legendrian curve with positive twisting number.  In \cite{honda2000classification} and \cite{honda2000classification2}, Honda classifies the tight contact structures on $M$, and in this note we classify the exact symplectic fillings of $M$, provided $\xi$ is virtually overtwisted and $t(S^1)<0$.

\begin{proposition}\label{prop:circle-bundles}
Let $M\to\Sigma$ be a circle bundle over a closed Riemann surface of genus $g>1$, and let $\xi$ be a virtually overtwisted tight contact structure on $M$ with $t(S^1)<0$.  Then $(M,\xi)$ admits a unique exact symplectic filling, up to symplectomorphism.
\end{proposition}

\begin{remark}
The only circle bundles over $S^2$ which admit virtually overtwisted contact structures have the form $L(|e|,1)$, where $e\leq -2$ is the Euler number of the circle bundle.  Any virtually overtwisted contact structure on such a lens space is uniquely exactly fillable, per Plamenevskaya--Van Horn-Morris \cite[Theorem 1.2]{plamenevskaya2010planar}, so the conclusion still holds.  In the $g=1$ case we have a circle bundle over $T^2$, which can also be realized as a parabolic torus bundle over $S^1$.  If $e\leq -2$, the conclusion again holds, but for $e\geq 2$ the virtually overtwisted structures admit no exact symplectic fillings.  See \cite[Theorem 1.1]{christian2021symplectic}.
\end{remark}

The only virtually overtwisted circle bundles not addressed by Proposition \ref{prop:circle-bundles}, the lens spaces treated in \cite{plamenevskaya2010planar}, or the torus bundles treated in \cite{christian2021symplectic} are those with $g>1$ and $t(S^1)=0$.  In \cite{lisca2003infinite,lisca2004tight} Lisca-Stipsicz verify a conjecture of Honda, which says that these structures are not symplectically semi-fillable, and thus are not symplectically fillable.  In Section~\ref{subsec:circle-bundle-proofs}, we will establish the following corollary.

\begin{corollary}\label{cor:circle-bundles}
Let $M\to\Sigma$ be a circle bundle over a closed Riemann surface, with virtually overtwisted tight contact structure $\xi$, and let $t(S^1)\leq 0$ be the twisting number.  If $t(S^1)=0$, then $(M,\xi)$ does not admit an exact symplectic filling; if $t(S^1)<0$, then $(M,\xi)$ admits a unique exact symplectic filling, up to symplectomorphism.
\end{corollary}

\subsection*{Acknowledgements} The authors thank an anonymous referee for many helpful comments, and thank Hyunki Min for a very useful correspondence.  The first author was partially supported by NSF grant DMS-1745583.  The second author was partially supported by Grant No. 11871332 of the National Natural Science Foundation of China.

\section{Proofs}

\subsection{Lens spaces}\label{sec:lens-space-proofs}
Throughout this section we will consider a lens space $L(p,q)$, $p>q>0$ with a virtually overtwisted contact structure as depicted in Figure \ref{fig:lens-space-filling}.  Namely,
\[
-\frac{p}{q}=[a_0,a_1,\ldots,a_k]
\]
for uniquely determined integers $a_i\leq -2$, and the stabilizations applied to the knots in Figure \ref{fig:lens-space-filling} do not all have the same sign.  We will prove Theorem \ref{thm:lens-space-fillings} by showing that every strong or exact symplectic filling of $(L(p,q),\xi)$ can be obtained by attaching a Weinstein 2-handle to a filling of a connected sum of the form
\[
(L(p',q'),\xi')\#(L(p'',q''),\xi''),
\]
obtained by deleting a single knot from the diagram describing $(L(p,q),\xi)$.  Beginning with an arbitrary filling of $(L(p,q),\xi)$, this decomposition may be inductively applied (in conjunction with Theorem \ref{thm:connected-sum}) until we have a symplectic filling of a connected sum of the form
\[
\mathop{\#}\limits_{i=1}^\ell (L(p_i,q_i),\xi_i),
\]
where each $(L(p_i,q_i),\xi_i)$ is a universally tight lens space.  To produce a complete list of the fillings of $(L(p,q),\xi)$, we consider the fillings of all connected sums of this form which may result from $(L(p,q),\xi)$.\\

In case one of the knots in Figure \ref{fig:lens-space-filling} has been stabilized both positively and negatively, we may directly apply Theorem \ref{thm:menke-knot}, as described in Section \ref{sec:intro}, to realize our symplectic filling as the result of attaching a Weinstein 2-handle to a connected sum.  We now focus on the case where no knots have been stabilized both positively and negatively.  In this case we may identify knots $K_+$ and $K_-$, each of which has been stabilized at least once, with all stabilizations being positive or negative, respectively.  Moreover, we may choose $K_+$ and $K_-$ to be adjacent, in that none of the knots between them have been stabilized.  Finally, our argument loses no generality by assuming that $K_+$ is to the right of $K_-$ in Figure \ref{fig:lens-space-filling}.\\

We now define
\[
-\frac{p'}{q'} = [a_0,\ldots,a_{k-1},a_k+1],
\]
where we identify $[a_0,\ldots,a_{k-1},a_k+1]$ with $[a_0,\ldots,a_{k-2},a_{k-1}+1]$ if $a_k=-2$.  Now \cite[Section 4.6]{honda2000classification} allows us to write $L(p,q)=V_0\cup_A V_1$, where $V_0$ and $V_1$ are solid tori with a map $A\colon\partial V_0\to\partial V_1$, the dividing curves of $\partial V_0$ are vertical, and the dividing curves of $\partial V_1$ have slope $-p'/q'$.  Moreover, we may decompose $V_1$ as
\[
V_1 = N \cup (V_1\setminus N),
\]
with $V_1\setminus N \cong T^2\times I$, such that $s_0=-1$ and $s_1=-p'/q'$.  Here we denote by $s_i$ the slope of the dividing curves of $T^2\times\{i\}$, for $i=0,1$.\\

The thickened torus $T^2\times I$ has a basic slice decomposition which we now describe.  Let
\[
0 \leq i_1 < i_2 < \cdots < i_\ell \leq k
\]
be the indices for which $a_{i_j}\leq -3$.  Then $T^2\times I$ decomposes into $\ell$ continued fraction blocks, with a total of
\[
|(a_{i_1}+2)(a_{i_2}+2)\cdots(a_{i_\ell}+2)|
\]
basic slices.  The basic slices in each continued fraction block will all be of a single sign, and the continued fraction blocks corresponding to $K_+$, $K_-$ will be adjacent, of opposite sign.  We immediately see that the boundary convex torus $T$ sitting between the continued fraction blocks associated to $K_+$ and $K_-$ is a mixed torus, sandwiched between basic slices $S_+\subset K_+$ and $S_-\subset K_-$ of opposite sign.\\

Let $-p'_1/q'_1$ be the slope of the dividing curves on $T$, and let $-p'_2/q'_2$, $-p'_0/q'_0$ be the opposite slopes of $S_+$, $S_-$, respectively.  We would like to normalize this neighborhood of $T$.  After observing that
\[
q'_1p'_2-p'_1q'_2 = 1
\quad\text{and}\quad
q'_0p'_1-p'_0q'_1 = 1,
\]
we see that applying the transformation
\[
\left(\begin{matrix}
1 & 0\\
p'_2q'_0-q'_2p'_0 - 1 & 1
\end{matrix}\right)
\left(\begin{matrix}
-p'_1 & -q'_1\\
p'_2 & q'_2
\end{matrix}\right)
\in SL(2,\mathbb{Z})
\]
leaves us with the slopes
\[
s_0 = -1,
\quad
s_1 = \infty,
\quad
s_2 = p'_2q'_0-q'_2p'_0 - 1.
\]
According to Theorem \ref{thm:jsj}, applying the JSJ decomposition to a filling of $(L(p,q),\xi)$ will produce a filling of $(M',\xi')$, obtained from $(L(p,q),\xi)$ by splitting with slope $0\leq s\leq p'_2q'_0-q'_2p'_0-2$ along $T$.\\

We now claim that $p'_2q'_0-q'_2p'_0-2=m+1$, where $m$ is the number of unstabilized knots between $K_+$ and $K_-$ in Figure \ref{fig:lens-space-filling}.  According to \cite[Lemma 4.12]{honda2000classification}, the slopes of the basic slice decomposition of $T^2\times I$ are obtained by incrementing the last entry of the continued fraction expansion of $-p'/q'$ until we have $-1$.  In particular, we may write
\[
-\frac{p'_2}{q'_2} = [a_0,a_1,\ldots,a_n,\overbrace{-2,\ldots,-2}^{m+1}]
\]
for some $n<k$ with $a_n\leq -3$ and then see that
\[
-\frac{p'_1}{q'_1} = [a_0,a_1,\ldots,a_n+1]
\]
and
\[
-\frac{p'_0}{q'_0} = [a_0,a_1,\ldots,a_n+2], \text{~if~} a_n\leq -4
\quad\text{or}\quad
-\frac{p'_0}{q'_0} = [a_0,a_1,\ldots,a_{n-1}+1], \text{~if~} a_n=-3.
\]
We can now verify our claim inductively.  If $a_n\leq -4$ we have
\[
[a_n+2] = -\frac{a_n+2}{-1}
\quad\text{and}\quad
[a_n,\overbrace{-2,\ldots,-2}^{m+1}] = -\frac{(m+2)a_n+(m+1)}{-(m+2)}
\]
and observe that
\[
(-1)((m+2)a_n + (m+1)) -(-(m+2))(a_n+2) = m+3.
\]
If we instead have $a_n=-3$, then
\[
[a_{n-1}+1] = -\frac{a_{n-1}+1}{-1}
\quad\text{and}\quad
[a_{n-1},-3,\overbrace{-2,\ldots,-2}^{m+1}] = -\frac{(2m+5)a_{n-1}+(m+2)}{-(2m+5)},
\]
so
\[
(-1)((2m+5)a_{n-1}+(m+2)) -(-(2m+5))(a_{n-1}+1)=m+3.
\]
In either case, we may now apply the following inductive step.  If $a/b$ and $a'/b'$ satisfy $ab'-a'b=m+3$, then
\[
[r,a/b] = \frac{ar-b}{a}
\quad\text{and}\quad
[r,a'/b'] = \frac{a'r-b'}{a'}
\]
satisfy
\[
(ar-b)a'-a(a'r-b') = ab'-ba'=m+3.
\]
This proves our claim, so we see that every filling of $(L(p,q),\xi)$ is obtained by attaching a round symplectic 1-handle to a filling of a contact manifold which is obtained from $(L(p,q),\xi)$ by splitting with slope $0\leq s\leq m+1$ along $T$.\\

\begin{figure}
\centering
\begin{tikzpicture}[scale=0.5]
\begin{knot}[
	clip width=5,
	clip radius=2pt,
	ignore endpoint intersections=false,
	flip crossing/.list={2,4,6}]
\strand (0,0) to[out=0,in=180]
        (2,0.5) to[out=0,in=180]
        (4,0) to[out=180,in=0]
        (3,-0.5) to[out=0,in=180]
        (4,-1) to[out=180,in=0]
        (2,-1.5) to[out=180,in=0]
        (0,-1) to[out=0,in=180]
        (1,-0.5) to[out=180,in=0]
        (0,0);
        
\strand (2.5,0) to[out=0,in=180]
        (6,1.25) to[out=0,in=180]
        (7,1) to[out=180,in=0]
        (6,0.5) to[out=0,in=180]
        (7,0) to[out=180,in=0]
        (6,-0.5) to[out=0,in=180]
        (7,-1) to[out=180,in=0]
        (6,-1.25) to[out=180,in=0]
        (2.5,0);
        
\strand (5.5,0) to[out=0,in=180]
		(8,0.75) to[out=0,in=180]
		(10.5,0) to[out=180,in=0]
		(8,-0.75) to[out=180,in=0]
		(5.5,0);

\strand (9,-1) to[out=0,in=180]
		(10,-0.5) to[out=180,in=0]
		(9,0) to[out=0,in=180]
		(10,0.5) to[out=180,in=0]
		(9,1) to[out=0,in=180]
		(10,1.25) to[out=0,in=180]
		(13.5,0) to[out=180,in=0]
		(10,-1.25) to[out=180,in=0]
		(9,-1);
\end{knot}

\draw[->] (7,-1.5) -- (7,-2.5);

\begin{scope}[yshift=-4cm]
\begin{knot}[
	clip width=5,
	clip radius=2pt,
	ignore endpoint intersections=false]
\strand[dashed] (-4,0) to[out=0,in=180]
		(-1.5,0.75) to[out=0,in=180]
		(1,0) to[out=180,in=0]
		(-1.5,-0.75) to[out=180,in=0]
		(-4,0);
\end{knot}

\node at (2,0) {$\sqcup$};
\begin{scope}[xshift=3cm]
\begin{knot}[
	clip width=5,
	clip radius=2pt,
	ignore endpoint intersections=false,
	flip crossing/.list={2,4,6}]
\strand[dashed] (0,0) to[out=0,in=180]
        (2,0.5) to[out=0,in=180]
        (4,0) to[out=180,in=0]
        (3,-0.5) to[out=0,in=180]
        (4,-1) to[out=180,in=0]
        (2,-1.5) to[out=180,in=0]
        (0,-1) to[out=0,in=180]
        (1,-0.5) to[out=180,in=0]
        (0,0);
        
\strand (2.5,0) to[out=0,in=180]
        (6,1.25) to[out=0,in=180]
        (7,1) to[out=180,in=0]
        (6,0.5) to[out=0,in=180]
        (7,0) to[out=180,in=0]
        (6,-0.5) to[out=0,in=180]
        (7,-1) to[out=180,in=0]
        (6,-1.25) to[out=180,in=0]
        (2.5,0);
        
\strand (5.5,0) to[out=0,in=180]
		(8,0.75) to[out=0,in=180]
		(10.5,0) to[out=180,in=0]
		(8,-0.75) to[out=180,in=0]
		(5.5,0);

\strand (9,-1) to[out=0,in=180]
		(10,-0.5) to[out=180,in=0]
		(9,0) to[out=0,in=180]
		(10,0.5) to[out=180,in=0]
		(9,1) to[out=0,in=180]
		(10,1.25) to[out=0,in=180]
		(13.5,0) to[out=180,in=0]
		(10,-1.25) to[out=180,in=0]
		(9,-1);
\end{knot}
\end{scope}

\end{scope}

\end{tikzpicture}
\caption{Every filling of the top lens space $L(89,24)$ with the given contact structure is obtained by attaching a round symplectic 1-handle to a filling of the disjoint union $S^3\sqcup L(24,7)$ below; the round 1-handle is attached along the dashed knots.  Fillings of $L(24,7)$ can be further decomposed as seen in Figure \ref{fig:example-part-two}.}
\label{fig:example}
\end{figure}

It is now straightforward to check that splitting with slope $s=0$ along $T$ produces a disjoint union of lens spaces, obtained from Figure \ref{fig:lens-space-filling} by deleting $K_-$ and realizing the two resulting chains of unknots in separate diagrams.  Attaching a round symplectic 1-handle to this disjoint union corresponds to first attaching a Weinstein 1-handle which produces the connected sum of these lens spaces, and then attaching a Weinstein 2-handle along $K_-$.  Similarly, splitting $(L(p,q),\xi)$ with slope $s=m+1$ along $T$ corresponds to deleting the knot $K_+$.  Each intermediate slope corresponds to deleting an unstabilized knot between $K_-$ and $K_+$.  In any case we see, as claimed above, that every filling of $(L(p,q),\xi)$ can be obtained by attaching a Weinstein 2-handle to a symplectic filling of a connected sum of lens spaces which is obtained by erasing a single knot from Figure \ref{fig:lens-space-filling}.  If the constituent lens spaces in this connected sum are virtually overtwisted, we may repeat this process until we have a connected sum of universally tight lens spaces.  This proves Theorem \ref{thm:lens-space-fillings}.\\

\begin{figure}
\centering
\begin{tikzpicture}[scale=0.5]
\begin{scope}[yshift=4cm]
\begin{knot}[
	clip width=5,
	clip radius=2pt,
	ignore endpoint intersections=false]
\strand[dashed] (-4,0) to[out=0,in=180]
		(-1.5,0.75) to[out=0,in=180]
		(1,0) to[out=180,in=0]
		(-1.5,-0.75) to[out=180,in=0]
		(-4,0);
\end{knot}

\node at (2,0) {$\sqcup$};
\begin{scope}[xshift=0.5cm]
\begin{knot}[
	clip width=5,
	clip radius=2pt,
	ignore endpoint intersections=false,
	flip crossing/.list={2,4,6}]
\strand[dashed] (2.5,0) to[out=0,in=180]
        (6,1.25) to[out=0,in=180]
        (7,1) to[out=180,in=0]
        (6,0.5) to[out=0,in=180]
        (7,0) to[out=180,in=0]
        (6,-0.5) to[out=0,in=180]
        (7,-1) to[out=180,in=0]
        (6,-1.25) to[out=180,in=0]
        (2.5,0);
        
\strand (5.5,0) to[out=0,in=180]
		(8,0.75) to[out=0,in=180]
		(10.5,0) to[out=180,in=0]
		(8,-0.75) to[out=180,in=0]
		(5.5,0);

\strand (9,-1) to[out=0,in=180]
		(10,-0.5) to[out=180,in=0]
		(9,0) to[out=0,in=180]
		(10,0.5) to[out=180,in=0]
		(9,1) to[out=0,in=180]
		(10,1.25) to[out=0,in=180]
		(13,0) to[out=180,in=0]
		(10,-1.25) to[out=180,in=0]
		(9,-1);
\end{knot}
\end{scope}
\end{scope}

\begin{scope}[yshift=0cm]
\begin{scope}[xshift=-9.5cm]
\begin{knot}[
	clip width=10,
	clip radius=2pt,
	ignore endpoint intersections=false,
	flip crossing/.list={2}]
\strand (2.5,0) to[out=0,in=180]
        (6,1.25) to[out=0,in=180]
        (7,1) to[out=180,in=0]
        (6,0.5) to[out=0,in=180]
        (7,0) to[out=180,in=0]
        (6,-0.5) to[out=0,in=180]
        (7,-1) to[out=180,in=0]
        (6,-1.25) to[out=180,in=0]
        (2.5,0);
        
\strand[dashed] (5.5,0) to[out=0,in=180]
		(8,0.75) to[out=0,in=180]
		(10.5,0) to[out=180,in=0]
		(8,-0.75) to[out=180,in=0]
		(5.5,0);
\end{knot}
\end{scope}

\node at (2,0) {$\sqcup$};

\begin{scope}[xshift=-2.5cm]
\begin{knot}[
	clip width=5,
	clip radius=2pt,
	ignore endpoint intersections=false,
	flip crossing/.list={2,4,6}]
\strand[dashed] (5.5,0) to[out=0,in=180]
		(8,0.75) to[out=0,in=180]
		(10.5,0) to[out=180,in=0]
		(8,-0.75) to[out=180,in=0]
		(5.5,0);
		
\strand (9,-1) to[out=0,in=180]
		(10,-0.5) to[out=180,in=0]
		(9,0) to[out=0,in=180]
		(10,0.5) to[out=180,in=0]
		(9,1) to[out=0,in=180]
		(10,1.25) to[out=0,in=180]
		(13.5,0) to[out=180,in=0]
		(10,-1.25) to[out=180,in=0]
		(9,-1);
\end{knot}
\end{scope}
\end{scope}

\begin{scope}[yshift=-4cm]
\begin{scope}[xshift=-12.5cm]
\begin{knot}[
	clip width=5,
	clip radius=2pt,
	ignore endpoint intersections=false,
	flip crossing/.list={2,4,6}]
\strand (2.5,0) to[out=0,in=180]
        (6,1.25) to[out=0,in=180]
        (7,1) to[out=180,in=0]
        (6,0.5) to[out=0,in=180]
        (7,0) to[out=180,in=0]
        (6,-0.5) to[out=0,in=180]
        (7,-1) to[out=180,in=0]
        (6,-1.25) to[out=180,in=0]
        (2.5,0);
        
\strand (5.5,0) to[out=0,in=180]
		(8,0.75) to[out=0,in=180]
		(10.5,0) to[out=180,in=0]
		(8,-0.75) to[out=180,in=0]
		(5.5,0);
		
\strand[dashed] (9,-1) to[out=0,in=180]
		(10,-0.5) to[out=180,in=0]
		(9,0) to[out=0,in=180]
		(10,0.5) to[out=180,in=0]
		(9,1) to[out=0,in=180]
		(10,1.25) to[out=0,in=180]
		(13.5,0) to[out=180,in=0]
		(10,-1.25) to[out=180,in=0]
		(9,-1);
\end{knot}
\end{scope}

\node at (2,0) {$\sqcup$};
\begin{scope}[xshift=-2.5cm]
\begin{knot}[
	clip width=5,
	clip radius=2pt,
	ignore endpoint intersections=false,
	flip crossing/.list={2,4,6}]
\strand[dashed] (5.5,0) to[out=0,in=180]
		(8,0.75) to[out=0,in=180]
		(10.5,0) to[out=180,in=0]
		(8,-0.75) to[out=180,in=0]
		(5.5,0);
\end{knot}
\end{scope}
\end{scope}
\end{tikzpicture}
\caption{Applying the JSJ decomposition to a filling of $L(24,7)$ with the contact structure seen in Figure \ref{fig:example} yields a filling of one of the three disjoint unions seen here.  We recover a filling of $L(24,7)$ by attaching a round symplectic 1-handle along the dashed knots.}
\label{fig:example-part-two}
\end{figure}

Theorem \ref{thm:lens-space-fillings} and its proof provide a recipe for classifying the fillings of a virtually overtwisted lens space $(L(p,q),\xi)$.  Given a depiction of the lens space as in Figure \ref{fig:lens-space-filling}, we can produce a tree whose leaves are disjoint unions of universally tight lens spaces, and every filling of $(L(p,q),\xi)$ can be obtained by attaching a specified sequence of round symplectic 1-handles to a filling of one of these disjoint unions.  An example of such a tree is given by taking Figures~\ref{fig:example} and \ref{fig:example-part-two} together.  The root of our tree is $(L(p,q),\xi)$, and we move to a new level of the tree by applying the decomposition described in this section.  If the mixed torus leading to the decomposition comes from a knot which has been stabilized both positively and negatively, we have a single branch.  If the mixed torus is associated to a pair $K_+,K_-$ of adjacent knots with opposite signs, then we have $m+2$ branches, where $m$ is the number of unstabilized knots between $K_+$ and $K_-$.\\

We observe that this argument recovers Fossati's classification of fillings for virtually overtwisted structures on lens spaces which result from contact surgery on the Hopf link (\cite[Theorem 1]{fossati2019contact}).  Consider $-\frac{p}{q}=[a_1,a_2]$, for some $a_1,a_2\leq -2$, and let $\xi_{\vot}$ be a virtually overtwisted contact structure on $L(p,q)$.  Our decomposition tells us that every filling of $(L(p,q),\xi_{\vot})$ is obtained by a specified Weinstein 2-handle attachment to a filling of either $L(-a_1,1)$ or $L(-a_2,1)$, with a particular (not necessarily virtually overtwisted) contact structure.  With the exception of a universally tight structure on $L(4,1)$, each lens space $L(-a_i,1)$ has a unique exact filling.  Moreover, we see from our decomposition that attaching a Weinstein 2-handle to such a standard filling in the manner prescribed will always yield the standard filling of $(L(p,q),\xi_{\vot})$.  So we have the following corollary.

\begin{corollary}[{c.f. \cite[Theorem 1]{fossati2019contact}}]
Let $(L(p,q),\xi_{\vot})$ be a virtually overtwisted lens space, with $-\frac{p}{q}=[a_1,a_2]$, for some $a_1,a_2\leq -2$.  Then $(L(p,q),\xi_{\vot})$ has
\begin{itemize}
	\item a unique exact filling, up to diffeomorphism, if $a_1\neq -4$ and $a_2\neq -4$, or if at least one of $a_1,a_2$ is $-4$ and the corresponding knot has been stabilized both positively and negatively;
	\item precisely two exact fillings, up to diffeomorphism, if at least one of $a_1,a_2$ is $-4$, and the corresponding knot has stabilizations of a single sign.
\end{itemize}
\end{corollary}

\subsection{Surgeries on Legendrian negative cables}
In this section we prove Theorem \ref{thm:cables}.  As in the statement of the theorem, we let $L\subset(S^3,\xi_{\std})$ be a Legendrian knot with smooth knot type $\mathcal{K}$, and let $Q(S_+S_-(L))$ be a Legendrian negative cable of $S_+S_-(L)$ with smooth knot type $\mathcal{K}_{p,q}$, $p<q(\tb(L)-2)$.  We suppose that the Thurston-Bennequin number of $Q(S_+S_-(L))$ is maximal among such knots, and we let $(M,\xi)$ be the contact manifold obtained by Legendrian surgery along $Q(S_+S_-(L))$.\\

We may use the stabilizations on $S_+S_-(L)$ to identify a mixed torus in $(M,\xi)$.  In particular, let $\nu(S_-(L))\subset (S^3,\xi_{\std})$ be a standard neighborhood of $S_-(L)$.  We let $V_1$ be the solid torus obtained from this neighborhood via Legendrian surgery along $Q(S_+S_-(L))$, and let $V_2=S^3\setminus \nu(S_-(L))$.  Then $M=V_1\cup V_2$, and we claim that the common boundary $\partial V_1=\partial \nu(S_-(L))=\partial V_2$ is a mixed torus.  Indeed, consider the three convex tori $\partial \nu(L)$, $\partial \nu(S_-(L))$, and $\partial \nu(S_+S_-(L))$.  The tori $\partial \nu(L)$ and $\partial \nu(S_-(L))$ cobound a negative basic slice in $M$.  The tori $\partial \nu(S_-(L))$ and $\partial \nu(S_+S_-(L))$ cobound a positive basic slice in $\nu(S_-(L))$, but this may not survive to a basic slice in $M$, since $Q(S_+S_-(L))$ may not be disjoint from $\partial \nu(S_+S_-(L))$.  However, we can subdivide this basic slice to find a boundary parallel convex torus cobounding a positive basic slice with $\partial \nu(S_-(L))$ (c.f. \cite[Lemma 3.15]{etnyre2001knots}).  So $\partial \nu(S_-(L))$ sits between basic slices of opposite sign, and is therefore a mixed torus.\\

Now because of our assumptions that $p<q(\tb(L)-2)$ and that the Thurston-Bennequin number of $Q(S_+S_-(L))$ is maximal, \cite[Theorem 5.16]{etnyre2018legendrian} tells us that $\tb(Q(S_+S_-(L))=pq$.  So the Legendrian surgery used to produce $V_1$ from $\nu(S_-(L))$ is smoothly $pq-1$-surgery.  According to Lemmas 7.2 and 7.3 of \cite{gordon1983dehn}, $V_1$ is then a solid torus $D^2\times S^1$ whose meridional curves have slope $(pq-1)/q^2$ in the coordinates of $\partial \nu(S_-(L))$ given by the meridian $\mu$ and the preferred longitude $\lambda$.  We now apply
\[
\left(\begin{matrix}
1 & 0\\ 1-\tb(L) & 1
\end{matrix}\right)\in SL(2,\mathbb{Z})
\]
to the coordinates of $\partial \nu(S_-(L))$.  In the original coordinates, the dividing curves of $\partial \nu(L)$ and $\partial \nu(S_-(L))$ had slopes $1/\tb(L)$ and $1/(\tb(L)-1)$, respectively.  In our new coordinates we find that $\Gamma_{\partial \nu(L)}$ has slope 1, $\Gamma_{\partial \nu(S_-(L))}$ is vertical, and the meridional slope of $\mu_{V_1}$ is represented by the vector $(pq-1+q^2(1-\tb(L)),q^2)$ in $\mathbb{Z}^2$.\\

\begin{figure}
\centering
\begin{tikzpicture}[scale=3]
\draw[thin] (0,-1) -- (0,1);
\draw[thin] (-1,0) -- (1,0);

\draw[thick] (-1,0) -- (1,0);
\node[above] at (1,0) {$\mu_S$};

\draw[thick] (-.707,-.707) -- (.707,.707);
\node[right] at (.707,.707) {$\Gamma_{\partial \nu(L)}$};

\draw[thick] (0,-1) -- (0,1);
\node[right] at (0,.8) {$\Gamma_{\partial \nu(S_-(L))}$};

\draw[thick, dashed] (-0.868,0.496) -- (0.868,-0.496);
\node[below] at (-0.868,0.496) {$\mu_{V_1}$};

\draw [thick,domain=0:150.26] plot ({0.6*cos(\x)}, {0.6*sin(\x)});
\draw [thick,domain=45:180] plot ({0.3*cos(\x)}, {0.3*sin(\x)});
\end{tikzpicture}
\caption{If the slope of $\mu_S$ were negative, then $V_1\cup S$ would be overtwisted; if the slope were positive, then $V_2\cup S$ would be overtwisted.  So $\mu_S$ is horizontal.}
\label{fig:slopes}
\end{figure}

Having made these preparations, we now suppose that $(W,\omega)$ is an exact symplectic filling of $(M,\xi)$.  Applying Theorem \ref{thm:jsj} to this filling yields $(W',\omega')$, an exact symplectic filling of its boundary $(M',\xi')$, which we may write as
\[
M' = M_1 \sqcup M_2 := (V_1\cup S) \sqcup (V_2\cup S),
\]
for some identifications $\partial S\to\partial V_i$, where $S$ is a solid torus.  The gluing maps $\partial S\to\partial V_i$ identify dividing curves, but the meridian $\mu_{S}$ of $S$ could in principle take any number of values.  Our first observation is that, because $\Gamma_{V_i}$ is vertical, $\mu_{S}=(1,m)\in\mathbb{Z}^2$ for some $m\in\mathbb{Z}$.  Next, the fact that $(M',\xi')$ is fillable means that each of $M_1$ and $M_2$ is tight.  On $M_1$, we see that as we move from the core of $S$ to $\partial V_1$ and then towards the core of $V_1$, the contact planes rotate from the slope of $\mu_S$ towards that of $\Gamma_{V_1}$, and finally towards the slope of $\mu_{V_1}$.  Because of our assumption that $p<q(\tb(L)-2)$, we find that $-1<\mu_{V_1}<0$.  Tightness demands that the total rotation of the contact planes is through an angle smaller than $\pi$, meaning that $m\geq 0$.  See Figure \ref{fig:slopes}.  On $M_2$ we see that the contact planes rotate counterclockwise from 1, the slope of $\Gamma_{\partial \nu(L)}$, to the slope of $\Gamma_{\partial \nu(S_-(L))}$, and finally to the slope $m$ of $\mu_S$.  Because this rotation must be smaller than $\pi$, we see that $m\leq 0$.  So we conclude that $m=0$.\\

Because the solid torus $S$ is attached with slope $m=0$, we find that $M_1=L(q^2,pq-1)$ and $M_2=V_2\cup S=S^3$.  Moreover, we see from the definition of $M_1$ that $M_1$ results from surgery on $(S^2\times S^1,\xi_{\std})$ along $\zeta(Q)$, as described in Section \ref{sec:intro}.  So $M_1\cong(L(q^2,pq-1),\xi_Q)$; on $S^3$ we have the unique tight contact structure $\xi_{\std}$.  Now Theorem \ref{thm:jsj} tells us that we recover $(W,\omega)$ from $(W',\omega')$ by attaching a round symplectic 1-handle along the cores of the two copies of $S$ --- one in $M_1$ and the other in $M_2$.  In $M_1$ this core is given by the image of $\{\mathrm{pt}\}\times S^1\subset(S^2\times S^1,\xi_{\std})$ after performing surgery along $\zeta(Q)\subset S^2\times S^1$.  That is, the core of $M_1$ is the image of $K$ after surgery, which we abusively call $K$.  In $M_2$ the core is given by $S_-(L)$.  We attach the round symplectic 1-handle by first attaching a Weinstein 1-handle along a pair of points $x\in K$ and $y\in S_-(L)$, and then attaching a Weinstein 2-handle along the resulting knot $K\#S_-(L)$.  So we obtain $(W,\omega)$ from $(W',\omega')$ by attaching a Weinstein 2-handle to
\[
(L(q^2,pq-1),\xi_Q)\#(S^3,\xi_{\std}) \cong (L(q^2,pq-1),\xi_Q)
\]
along $K\# S_-(L)$.  Since this is precisely the knot $L_Q$ identified in the statement of Theorem~\ref{thm:cables}, our proof is complete.

\subsection{Seifert fibered spaces over $S^2$}\label{subsec:sfs-proofs}
Before proceeding to the proofs of the results in Section~\ref{subsec:sfs}, we first recall in Section~\ref{subsubsec:mixed-structures} what it means for a tight contact structure $\xi$ on a Seifert fibered space $M=M(\frac{q_1}{p_1},\cdots,\frac{q_n}{p_n})$ to be thoroughly or lightly mixed, and we identify the universally tight contact structures on small Seifert fibered spaces.  Sections~\ref{subsubsec:positive-euler} and~\ref{sec:negative-euler} then contain the proofs in the $e_0\geq 0$ and $e_0\leq -3$ cases, respectively.

\subsubsection{Mixed contact structures on Seifert fibered spaces}\label{subsubsec:mixed-structures}
As in Section \ref{sec:intro}, we take $n\geq 3$, $q_i,p_i>0$ coprime, and assume that $q_i<p_i$ for $i=1,\ldots,n-1$.  We also have continued fraction expansions as in (\ref{eq:continued-fraction}), and we denote by $e_0=\lfloor\frac{q_n}{p_n}\rfloor$ the Euler number of $M$.\\

To accommodate for the fact that we may have $q_n>p_n$, we introduce auxiliary coefficients $b_0^n,\ldots,b_{l'_n}^n$ defined by
\[
-\frac{p_n}{q_n}=[a_{0}^{n}, a_{1}^{n}, ..., a_{l_{n}}^{n}] = [-1, -2, ..., -2, b_{0}^{n}-1,  b_{1}^{n}, ..., b_{l'_{n}}^{n}],
\]
where $l'_n=l_n-e_0$, meaning that the number of $-2$s preceding $b_0^n-1$ is $e_0-1$.\\

In Section~\ref{subsec:sfs}, we defined thoroughly mixed tight contact structures via surgery diagrams; we now present these structures as those which result from a particular construction.  We let $\Sigma$ be a planar surface with $n$ boundary components, and write
\[
-\partial(\Sigma\times S^1) = T_1 + T_2 + \cdots + T_n
\]
for the torus boundary components of $\Sigma\times S^1$.  Now let $\xi$ be an $S^1$-invariant, virtually overtwisted tight contact structure on $\Sigma\times S^1$ such that
\begin{enumerate}[label=(\arabic*)]
	\item each $T_i$ is a minimal convex torus, with dividing curves of slope $-1$ for $i<n$ and slope $-e_0-1$ for $i=n$;
	\item adjacent to each $T_i$ is a positive basic slice $L_i$, with $\partial L_i= T_i-T'_i$;
	\item each $T'_i$ is a minimal convex torus, with dividing curves of slope $\infty$.
\end{enumerate}
Such a contact structure exists by \cite[Section 5]{honda2000classification2}.\\

\begin{figure}
\centering
\begin{tikzpicture}
\begin{scope}[scale=1.25]
\draw (0,0) ellipse (6 and 3);
\draw (0,0) ellipse (5.5 and 2.5);
\node[above] at (0,-2.5) {$\infty$};
\node[below] at (0,-2.5) {$+$};
\node[below] at (0,-3) {$-e_0-1$};
\node[above] at (2,-2.375) {$T_n'$};
\node[below] at (2,-2.875) {$T_n$};
\end{scope}

\begin{scope}[xshift=-4cm,scale=0.85]
\draw (0,0) circle (0.75);
\draw (0,0) circle (1.15);
\node[below] at (0,0.75) {$-1$};
\node[above] at (0,-0.75) {$T_1$};
\node[above] at (0,-1.2) {$+$};
\node[below] at (0,-1.15) {$T_1'$};
\node[above] at (0,1.15) {$\infty$};
\end{scope}

\begin{scope}[xshift=-1cm,scale=0.85]
\draw (0,0) circle (0.75);
\draw (0,0) circle (1.15);
\node[below] at (0,0.75) {$-1$};
\node[above] at (0,-0.75) {$T_2$};
\node[above] at (0,-1.2) {$+$};
\node[below] at (0,-1.15) {$T_2'$};
\node[above] at (0,1.15) {$\infty$};
\end{scope}

\begin{scope}[xshift=4cm,scale=0.85]
\draw (0,0) circle (0.75);
\draw (0,0) circle (1.15);
\node[below] at (0,0.75) {$-1$};
\node[above] at (0,-0.75) {$T_{n-1}$};
\node[above] at (0,-1.2) {$+$};
\node[below] at (0,-1.15) {$T_{n-1}'$};
\node[above] at (0,1.15) {$\infty$};
\end{scope}

\filldraw[color=black, fill=black](0.5,0) circle (0.1);
\filldraw[color=black, fill=black](1.5,0) circle (0.1);
\filldraw[color=black, fill=black](2.5,0) circle (0.1);
\end{tikzpicture}
\caption{The first layer of basic slices attached to $\Sigma\times S^1$.}
\label{fig:sigma}
\end{figure}

For each $i=1,\ldots,n-1$, we will attach $-2-a_0^i$ basic slices to $(\Sigma\times S^1,\xi)$, with slopes
\[
-1,-\frac{1}{2},-\frac{1}{3},\ldots,\frac{1}{a_0^i+1},
\]
starting at $T_i$.  Similarly, we attach $-2-b_0^n$ basic slices, starting from $T_n$, with slopes
\[
-e_0-1,-e_0-\frac{1}{2},\ldots,-e_0+\frac{1}{b_0^n+1}.
\]
For $i=1,\ldots,n$, we call the boundary of the outermost basic slice $T_i''$.  Finally, we let $V_i$ be a solid torus and choose a tight contact structure on $V_i$ such that $\partial V_i$ is minimal, convex, and has dividing curves of slope
\[
[a_{l_i}^i,a_{l_i-1}^i,\ldots,a_2^i,a_1^i+1]
\]
for $1\leq i \leq n-1$, and slope
\[
[b_{l_n'}^n,b_{l_n'-1}^n,\ldots,b_2^n,b_1^n+1]
\]
for $i=n$.  Notice that there are $\vert\prod_{j=1}^{l_i}(a_j^i+1)\vert$ (respectively, $\vert\prod_{j=1}^{l_n'}(b_j^n+1)\vert$) such tight structures on $V_i$, per Honda's classification \cite{honda2000classification}.  We then attach each $V_i$ to $\Sigma\times S^1$ by identifying the dividing curves and meridians of $\partial V_i$ with those of $T_i''$.  The result is a tight contact structure on $M$, and we call any structure resulting from this construction \emph{thoroughly mixed}.  It is not difficult to check that these are precisely the structures identified in Section~\ref{subsec:sfs}.\\

Note that this construction is not unique.  For instance, we may shuffle the order in which we attach basic slices within a given continued fraction block without changing our contact structure.  But the important feature is that by ensuring that the innermost basic slice around each boundary component is positive, we may find $n$ mixed tori.

\begin{lemma}\label{lemma:mixed}
In a thoroughly mixed tight contact structure, each torus $T_i'$, $1\leq i\leq n$, is a mixed torus with vertical dividing curves.
\end{lemma}
\begin{proof}
We show that $T_n'$ is mixed; the other tori are similar.  In $\Sigma\times S^1$, consider a collection $A_1,\ldots,A_{n-2}$ of vertical annuli as in Figure \ref{fig:thoroughly-mixed}, with $A_i$ connecting $T_i$ to $T_{i+1}$.  Each annulus will have parallel horizontal dividing curves, and we consider the neighborhood
\[
N=N(T_1\cup\cdots\cup T_{n-1}\cup A_1\cup\cdots\cup A_{n-2}),
\]
whose boundary is given by $\partial N=T_1\cup\cdots\cup T_{n-1}\cup T'$.  Here $T'$ has dividing curves of slope $1$, measured in the coordinates of $T_n$.  Because each of the basic slices $L_i$ is positive, the toric annulus
\[
(\Sigma\times S^1)\setminus(N\cup L_n)
\]
is a negative basic slice with boundary slopes $\infty$ and $1$.  So $T_n'$ is sandwiched between basic slices of opposite sign whose slopes are $-e_0-1,\infty$, and $1$, meaning that $T_n'$ is a mixed torus.
\end{proof}

\begin{figure}
\centering
\begin{tikzpicture}
\begin{scope}[scale=1.25]
\draw (0,0) ellipse (6 and 3);
\draw (0,0) ellipse (5.5 and 2.5);
\draw[dashed] (0,0) ellipse (5 and 2);
\node[above] at (0,-2.5) {$\infty$};
\node[below] at (1,-2.5) {$+$};
\node[below] at (0,-3) {$-e_0-1$};
\node[above] at (2,-2.375) {$T_n'$};
\node[below] at (2,-2.875) {$T_n$};
\node[above] at (1,-2.5) {$-$};
\node[above] at (0,-2) {$1$};
\node[above] at (2,-1.875) {$T'$};
\end{scope}

\begin{scope}[xshift=-4cm,scale=0.85]
\draw (0,0) circle (0.75);
\draw (0,0) circle (1.15);
\node[below] at (0,0.75) {$-1$};
\node[above] at (0,-0.75) {$T_1$};
\node[above] at (0,-1.2) {$+$};
\node[below] at (0,-1.15) {$T_1'$};
\node[above] at (0,1.15) {$\infty$};
\draw[dashed] (0.75,0) -- (2.75,0);
\node[above] at (1.75,0) {$A_1$};
\end{scope}

\begin{scope}[xshift=-1cm,scale=0.85]
\draw (0,0) circle (0.75);
\draw (0,0) circle (1.15);
\node[below] at (0,0.75) {$-1$};
\node[above] at (0,-0.75) {$T_2$};
\node[above] at (0,-1.2) {$+$};
\node[below] at (0,-1.15) {$T_2'$};
\node[above] at (0,1.15) {$\infty$};
\draw[dashed] (0.75,0) -- (1.75,0);
\node[above] at (1.35,0) {$A_2$};
\end{scope}

\begin{scope}[xshift=4cm,scale=0.85]
\draw (0,0) circle (0.75);
\draw (0,0) circle (1.15);
\node[below] at (0,0.75) {$-1$};
\node[above] at (0,-0.75) {$T_{n-1}$};
\node[above] at (0,-1.2) {$+$};
\node[below] at (0,-1.15) {$T_{n-1}'$};
\node[above] at (0,1.15) {$\infty$};
\draw[dashed] (-0.75,0) -- (-1.75,0);
\node[above left] at (-1,0) {$A_{n-2}$};
\end{scope}

\filldraw[color=black, fill=black](1,0) circle (0.05);
\filldraw[color=black, fill=black](1.5,0) circle (0.05);
\filldraw[color=black, fill=black](2,0) circle (0.05);
\end{tikzpicture}
\caption{Each torus $T_i'$ is mixed.}
\label{fig:thoroughly-mixed}
\end{figure}

Recall that we also defined \emph{lightly mixed} tight contact structures in Section~\ref{subsec:sfs}; for convenience, we repeat the definition here.

\begin{definition}
Let $\xi$ be a tight contact structure on $M(\frac{q_1}{p_1},\cdots,\frac{q_n}{p_n})$.  We will call $\xi$ \emph{lightly mixed} if $\xi$ is not thoroughly mixed, but admits a Stein filling as in Figure \ref{fig:seifert-filling} for which at least $n-2$ of $K_1,\ldots,K_n$ have been stabilized both positively and negatively.  We say that $\xi$ is \emph{lightly mixed about $K_i$ and $K_j$} to indicate that $\xi$ admits a Stein filling for which each of $K_1,\ldots,K_n$ except $K_i$ and $K_j$ have been stabilized positively and negatively.
\end{definition}

Consider the tight contact structures on a small Seifert manifold $M=M(\frac{q_1}{p_1},\frac{q_2}{p_2},\frac{q_3}{p_3})$, with $p_i,q_i>0$ chosen as above, so that $e_0\geq 0$.  According to \cite{ghiggini2006classification}, each of these can be represented as in Figure \ref{fig:seifert-filling}.  Let $K_3'$ be the nearest unknot adjacent to $K_3$ which has been stabilized --- meaning that $K_3'=K_3$ if $e_0=0$.  If each of $K_1,K_2,K_3'$ has been stabilized positively at least once (or, according to the classification in \cite{ghiggini2006classification}, if each has been stabilized negatively at least once), then the tight contact structure is thoroughly mixed.  On the other hand, if one of $K_1,K_2,K_3$ has been stabilized both positively and negatively while the other two have stabilizations of a single sign (the signs on the two knots being opposite), then the tight structure is lightly mixed.  This leaves precisely $6\vert\Pi_{i=1}^3\Pi_{j=1}^{l_i}(a_j^i+1)\vert$ tight contact structures on $M$ which are neither lightly nor thoroughly mixed.  In these structures, each of $K_1,K_2,K_3'$ has all of its stabilizations of a single sign, but the three knots do not all use the same sign.  If the stabilizations of adjacent knots in Figure \ref{fig:seifert-filling} always match, then the following lemma says that we have a universally tight contact structure; note that there are precisely six such structures.

\begin{figure}
\centering
\begin{tikzpicture}
\begin{scope}[scale=1.25]
\draw[dashed] (0,0) ellipse (5 and 3);
\draw (0,0) ellipse (4.5 and 2.5);
\draw (0,0) ellipse (4 and 2);
\node[above] at (0,-2.5) {$s_3$};
\node[below] at (0,-3) {$0$};
\node[above] at (2,-2.25) {$T_3$};
\node[above] at (1,-2.5) {$-$};
\node[above] at (0,-2) {$\infty$};
\node[above] at (2,-1.75) {$T_3'$};
\end{scope}

\begin{scope}[xshift=-2cm,scale=1.15]
\draw (0,0) circle (0.75);
\draw (0,0) circle (1.15);
\draw[dashed] (0,0) circle (0.35);
\node[below] at (0,0.35) {$0$};
\node[above] at (0,0.75) {$s_1$};
\node[above] at (0,-0.85) {$T_1$};
\node[above] at (0,-1.2) {$+$};
\node[below] at (0,-1.15) {$T_1'$};
\node[above] at (0,1.15) {$\infty$};
\end{scope}

\begin{scope}[xshift=2cm,scale=1.15]
\draw (0,0) circle (0.75);
\draw (0,0) circle (1.15);
\draw[dashed] (0,0) circle (0.35);
\node[below] at (0,0.35) {$0$};
\node[above] at (0,0.75) {$s_2$};
\node[above] at (0,-0.85) {$T_2$};
\node[above] at (0,-1.2) {$+$};
\node[below] at (0,-1.15) {$T_2'$};
\node[above] at (0,1.15) {$\infty$};
\end{scope}
\end{tikzpicture}
\caption{The surface $\Sigma\times S^1$ sits inside of a Seifert fibered space $M(\frac{q_1}{p_1},\frac{q_2}{p_2},\frac{q_3}{p_3})$ which is neither lightly nor thoroughly mixed.  This surface may be extended to $\tilde{\Sigma}\times S^1$, whose boundary components have horizontal dividing curves.}
\label{fig:universally-tight}
\end{figure}

\begin{lemma}\label{lemma:ut}
Let $\xi$ be a tight contact structure on $M(\frac{q_1}{p_1},\frac{q_2}{p_2},\frac{q_3}{p_3})$ for some $0<q_i,p_i$, with $q_i<p_i$ for $i=1,2$, which is neither lightly mixed nor thoroughly mixed.  If each of the horizontal links in Figure \ref{fig:seifert-filling} has stabilizations of only one sign, then $\xi$ is universally tight.
\end{lemma}
\begin{proof}
Notice that the Euler number of $M$ satisfies $e_0\geq 0$.  Per the classification of tight contact structures due to Wu \cite{wu2004tight} and Ghiggini-Lisca-Stipsicz \cite{ghiggini2006classification} on such Seifert fibered spaces, we may write
\[
M = M\left(\frac{q_1}{p_1},\frac{q_2}{p_2},\frac{q_3}{p_3}\right) \cong (\Sigma\times S^1) \cup_{(\varphi_1\cup\varphi_2\cup\varphi_3)} (V_1\cup V_2\cup V_3),
\]
where each $V_i$ is a soid torus, $-\partial(\Sigma\times S^1)=T_1+T_2+T_3$, and $\varphi_i\colon\partial V_i\to T_i$ is an orientation-preserving diffeomorphism.  Moreover, we may take $s_i$, the slope of the dividing curves of $T_i=\partial V_i$ in the coordinates of $T_i$, to satisfy
\[
\frac{1}{a_0^i+1}<s_i<-\frac{q_i}{p_i} \text{~for~} i=1,2,
\quad\text{and}\quad
-e_0+\frac{1}{b_0^3+1} < s_3 < -\frac{q_3}{p_3}.
\]
Here $a_0^i$ and $b_0^3$ are as above.  In particular, we have $s_1,s_2\in(-1,0)$ and $s_3\in(-e_0-1,-e_0)$.\\

Continuing to follow \cite[Section 3.3]{wu2004tight}, we may thicken each $V_i$ to a solid torus $V_i'$ such that $T_i':=\partial V_i'$ is a minimal convex torus with vertical dividing curves when measured in the coordinates of $T_i$.  Now $V_i'\setminus V_i$ is a toric annulus bounded by $T_i$ and $T_i'$ which we may factor into basic slices.  Because $\xi$ fails to be thoroughly or lightly mixed, all of the basic slices between $T_i$ and $T_i'$ must have the same sign, but the signs for $i=1,2,3$ are not all the same.  For instance, Figure \ref{fig:universally-tight} depicts a case where the basic slices between $T_i$ and $T_i'$ are positive for $i=1,2$, but negative for $i=3$.  Now consider attaching basic slices of matching sign to each $T_i$ until we obtain $\tilde{\Sigma}\times S^1$, whose boundary components all have horizontal dividing curves.  According to \cite[Lemma 5.1]{honda2000classification2}, $\tilde{\Sigma}\times S^1$ is universally tight; because $\tilde{\Sigma}\times S^1$ contains $\Sigma\times S^1$, we see that $\Sigma\times S^1$ is also universally tight.  Each solid torus $V_i$ is universally tight because the stabilizations used to produce the tight structure on $V_i$ are all of one sign.  Moreover, the homomorphism
\[
i_*\colon\pi_1(\Sigma\times S^1) \to \pi_1(M)
\]
induced by inclusion is a surjection.  We conclude that $M$ is universally tight.
\end{proof}

Following this observation, Corollary \ref{corollary:vot} follows from Corollaries \ref{cor:thorough} and \ref{cor:lightly}.\\

Finally, we consider the remaining contact structures on $M(\frac{q_1}{p_1},\frac{q_2}{p_2},\frac{q_3}{p_3})$ --- those which are neither thoroughly nor lightly mixed, and to which Lemma \ref{lemma:ut} does not apply.  All such contact structures are virtually overtwisted.

\begin{lemma}\label{lemma:vot}
Let $\xi$ be a tight contact structure on $M(\frac{q_1}{p_1},\frac{q_2}{p_2},\frac{q_3}{p_3})$, with surgery diagram as in Figure \ref{fig:seifert-filling}.  If any of the horizontal links have both positive and negative stabilizations, then $\xi$ is virtually overtwisted.
\end{lemma}

The proof of Lemma \ref{lemma:vot} will make use of the following topological fact about small Seifert fibered spaces.

\begin{lemma}\label{lemma:cover}
Any small Seifert fibered space $M$ admits a finite sheeted cover $\tilde{M}$ such that the induced Seifert fibration on $\tilde{M}$ has no exceptional fibers.
\end{lemma}
\begin{proof}
In case the fundamental group $\pi_1(M)$ is infinite, this follows from \cite[Lemma 2.4.22]{brin2007seifert}, so we focus on the case where $\pi_1(M)$ is finite.  The universal cover of a small Seifert fibered space with finite fundamental group is $S^3$, so we have a diagram
\[
\begin{tikzcd}
S^3 \arrow{r}{p} \arrow{d}{\tilde{\pi}} & M \arrow{d}{\pi}\\
\tilde{B} \arrow{r}{\overline{p}} & B
\end{tikzcd},
\]
where $B$ is $S^2$ with three cone points, $\pi\colon M\to B$ is the Seifert fibration on $M$, $p\colon S^3\to M$ is the covering map, and $\tilde{\pi}\colon S^3\to\tilde{B}$ is the induced Seifert fibration on $S^3$.  Because $\pi_1(M)$ is finite, we have $B=S^2(a,b,c)$ for some $(a,b,c)\in\{(2,2,n),(2,3,3),(2,3,4),(2,3,5)\}$.  That is, $M$ is a platonic Seifert fibered space.  The map $\overline{p}\colon\tilde{B}\to B$ is an orbifold covering map.  We notice that since $B$ has positive orbifold characteristic, the same is true of $\tilde{B}$, and also that $\tilde{B}$ has at most two cone points, since $\tilde{B}$ is the base of a Seifert fibration of $S^3$.  So $\overline{p}$ is a positive orbifold covering map of the form $S^2(a',b')\to S^2(a,b,c)$; such maps are classified by \cite[Proposition 5.5]{boyle2018virtual}, from which we conclude that $\tilde{B}=S^2(d,d)$ for some $d\geq 1$.\\

At the same time, we use \cite[Proposition 5.2]{geiges2018seifert} to write the Seifert fibration $\tilde{\pi}$ as
\[
M(0;(\alpha_1,\beta_1),(\alpha_2,\beta_2)),
\]
for some natural numbers $\alpha_1\geq \alpha_2$ and integers $\beta_1,\beta_2$ satisfying $0\leq\beta_1<\alpha_1$ and $\alpha_1\beta_2+\beta_1\alpha_2=1$.  The base of this Seifert fibration is $S^2(\alpha_1,\alpha_2)$, so we conclude that $\alpha_1=\alpha_2=d\geq 1$.  But this means that $d\beta_2+\beta_1d=1$, so we must have $d=1$.  We conclude that $\tilde{B}=S^2(1,1)$ has no cone points, and thus $\tilde{\pi}$ has no exceptional fibers.
\end{proof}

\begin{proof}[Proof of Lemma \ref{lemma:vot}]
Let us decompose $M:=M(\frac{q_1}{p_1},\frac{q_2}{p_2},\frac{q_3}{p_3})$ as in the proof of Lemma \ref{lemma:ut}, writing
\[
M = (\Sigma'\times S^1)\cup_{(\varphi_1\cup\varphi_2\cup\varphi_3)}(V_1'\cup V_2'\cup V_3'),
\]
where $-\partial(\Sigma'\times S^1)=T_1'+T_2'+T_3'$, and the dividing curves on each $T_i'$ have slope $\infty$.  For $i=1,2,3$, we may express the orientation-preserving diffeomorphism $\varphi_i\colon\partial V_i\to T_i$ via
\[
\varphi_i = \left(\begin{matrix}
p_i & -u_i\\ -q_i & v_i
\end{matrix}\right),
\]
for some $u_i,v_i$ satisfying $p_iv_i-q_iu_i=1$.  In the coordinates of $\partial V_i'$, the dividing curves thus have slope represented by
\[
\varphi_i^{-1}\left(\begin{matrix}0\\ 1\end{matrix}\right) = \left(\begin{matrix}
v_i & u_i\\ q_i & p_i
\end{matrix}\right)\left(\begin{matrix}0\\ 1\end{matrix}\right) = \left(\begin{matrix} u_i\\ p_i \end{matrix}\right).
\]
So $V_i'$ is a solid torus whose boundary has dividing curves of slope $p_i/u_i$, for $i=1,2,3$.  If $V_i'$ is virtually overtwisted, then lifting $\xi|_{V_i'}$ via the $p_i$-fold cover $\tilde{V}_i' \to V_i'$ produces an overtwisted contact structure on $\tilde{V}_i'$.  (See, for example, \cite[Exercise 6.45]{etnyre2004convex}.)  Now Lemma \ref{lemma:cover} allows us to construct a finite sheeted cover $p\colon\tilde{M}\to M$ such that $V_i'$ lifts to several copies of $\tilde{V}_i'$, for $i=1,2,3$.  Because $(M,\xi)$ has a horizontal link with both positive and negative stabilizations, at least one of $V_1',V_2'$, and $V_3'$ is virtually overtwisted, and thus a lift of this solid torus in $\tilde{M}$ is overtwisted.  We conclude that $(\tilde{M},p^*\xi)$ is overtwisted, and thus $(M,\xi)$ is virtually overtwisted.
\end{proof}

Lemma \ref{lemma:vot} applies to any contact structure which is neither thoroughly nor lightly mixed, and to which Lemma \ref{lemma:ut} does not apply.  Lemma \ref{lemma:vot} also applies to all lightly mixed contact structures and, if $e_0>0$, all but two thoroughly mixed contact structures.  If $e_0=0$ and $M(\frac{q_1}{p_1},\frac{q_2}{p_2},\frac{q_3}{p_3})$ does not have $q_i=p_i-1$ for $i=1,2,3$, then each horizontal link in Figure \ref{fig:seifert-filling} has more than one stabilization, and the classification of tight contact structures \cite[Theorem 2.7]{ghiggini2006classification} allows us to change the sign of one stabilization on each horizontal link to ensure that Lemma \ref{lemma:vot} applies to at least one of these links.  On the other hand, if $q_i=p_i-1$ for $i=1,2,3$, then each horizontal link has exactly one stabilization, so Lemma \ref{lemma:vot} does not apply.  Altogether, we see that if $e_0>0$ there are at most 8 universally tight contact structures on $M(\frac{q_1}{p_1},\frac{q_2}{p_2},\frac{q_3}{p_3})$, while if $e_0=0$, there are at most 7 universally tight contact structures.  If $e_0=0$ and $q_i\neq p_i-1$ for some $i=1,2,3$, then there are precisely 6 universally tight contact structures on $M(\frac{q_1}{p_1},\frac{q_2}{p_2},\frac{q_3}{p_3})$.

\subsubsection{The case $e_0\geq 0$}\label{subsubsec:positive-euler}
Theorem \ref{thm:lightly-mixed} is a straightforward consequence of Theorem 1.3 of \cite{menke2018jsj} and the definition of lightly mixed contact structures, so we prove this result first.\\

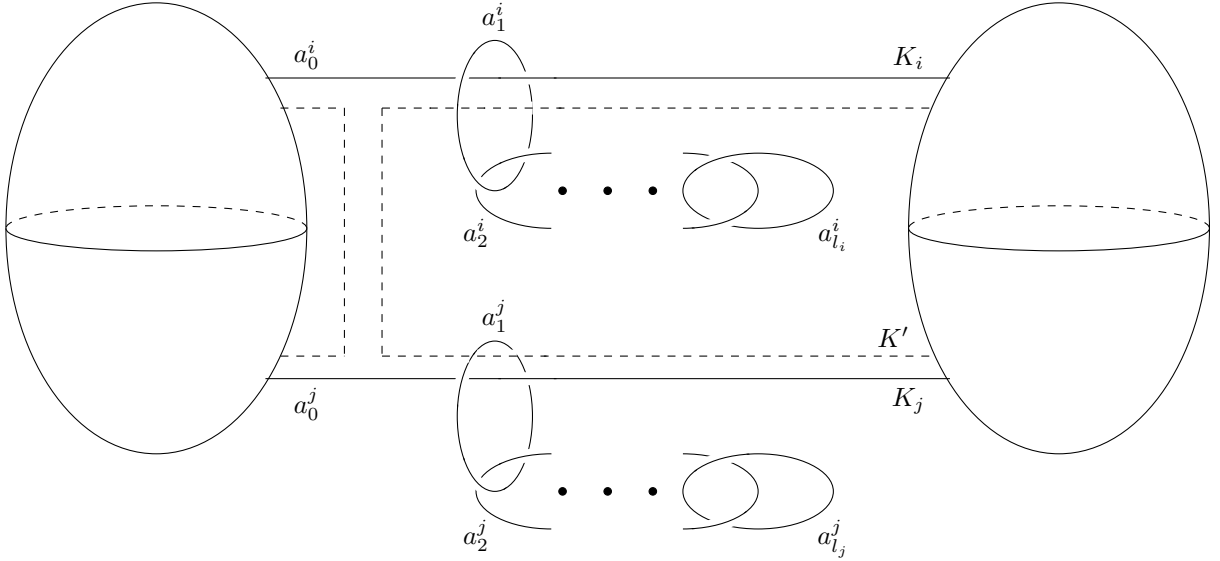
\begin{figure}
\centering
\begin{tikzpicture}
\draw (-6,0) ellipse (2 and 3);
\draw (-8,0) arc (180:360:2 and 0.3);
\draw[dashed] (-4,0) arc (0:180:2 and 0.3);
\draw (6,0) ellipse (2 and 3);
\draw (4,0) arc (180:360:2 and 0.3);
\draw[dashed] (8,0) arc (0:180:2 and 0.3);

\begin{scope}[yshift=-1cm]
\begin{knot}[
	clip width=3.5,
	ignore endpoint intersections=false,
	flip crossing/.list={1,5,8}]
\strand (-4.55,3) to (4.55,3);
\strand[dashed] (-3,2.6) to (4.35,2.6);
\strand[dashed] (-4.35,2.6) to (-3.5,2.6);
\strand (-1,2.5) arc (360:0:0.5 and 1);
\strand (-0.75,1) arc (270:90:1 and 0.5);
\strand (1,1) arc (-90:90:1 and 0.5);
\strand (3,1.5) arc (0:360:1 and 0.5);
\end{knot}
\end{scope}

\draw[dashed] (-3,1.6) -- (-3,-1.7);
\draw[dashed] (-3.5,1.6) -- (-3.5,-1.7);

\begin{scope}[yshift=-5cm]
\begin{knot}[
	clip width=3.5,
	ignore endpoint intersections=false,
	flip crossing/.list={1,5,8}]
\strand (-4.55,3) to (4.55,3);
\strand[dashed] (-3,3.3) to (4.35,3.3);
\strand[dashed] (-4.35,3.3) to (-3.5,3.3);
\strand (-1,2.5) arc (360:0:0.5 and 1);
\strand (-0.75,1) arc (270:90:1 and 0.5);
\strand (1,1) arc (-90:90:1 and 0.5);
\strand (3,1.5) arc (0:360:1 and 0.5);
\end{knot}
\end{scope}

\begin{scope}[yshift=5cm]
\filldraw[color=black, fill=black](0,-4.5) circle (0.05);
\filldraw[color=black, fill=black](-0.6,-4.5) circle (0.05);
\filldraw[color=black, fill=black](0.6,-4.5) circle (0.05);
\end{scope}
\begin{scope}[yshift=1cm]
\filldraw[color=black, fill=black](0,-4.5) circle (0.05);
\filldraw[color=black, fill=black](-0.6,-4.5) circle (0.05);
\filldraw[color=black, fill=black](0.6,-4.5) circle (0.05);
\end{scope}

\begin{scope}[yshift=-1cm]
\node[above] at (4,3) {$K_{n-1}$};
\node[above] at (-4,3) {$a_0^{n-1}$};
\node[above] at (-1.5,3.5) {$a_1^{n-1}$};
\node[below] at (-1.75,1.25) {$a_2^{n-1}$};
\node[below] at (3,1.25) {$a_{l_i}^{n-1}$};
\end{scope}

\begin{scope}[yshift=-5cm]
\node[above] at (3.8,3.3) {$K'$};
\node[below] at (4,3) {$K_n$};
\node[below] at (-4,3) {$a_0^n$};
\node[above] at (-1.5,3.5) {$a_1^n$};
\node[below] at (-1.75,1.25) {$a_2^n$};
\node[below] at (3,1.25) {$a_{l_j}^n$};
\end{scope}
\end{tikzpicture}
\caption{In the handlebody diagram for $(M_{n-2},\zeta_{n-2})$, both $K_{n-1}$ and $K_n$ pass over the 1-handle.  To realize $(M_{n-2},\zeta_{n-2})$ as a lens space, we slide $K_n$ over $K_{n-1}$ to produce $K'$, which has framing $a_0^{n-1}+a_0^n$, and then cancel $K_{n-1}$ with the 1-handle.}
\label{fig:light-diagrams}
\end{figure}

Suppose that $M=M(\frac{q_1}{p_1},\cdots,\frac{q_n}{p_n})$ is a Seifert fibered space, for some $n\geq 3$ and coprime positive integers $q_i,p_i$.  If $\xi$ is a lightly mixed tight contact structure on $M$, then we may realize $(M,\xi)$ as the boundary of a Stein handlebody as in Figure \ref{fig:seifert-filling}, with $n-2$ of the horizontal knots $K_1,\ldots,K_n$ having been stabilized both positively and negatively.  Without loss of generality, we may assume that each of $K_1,\ldots,K_{n-2}$ has been stabilized both positively and negatively.  Notice that $(M,\xi)$ is obtained from the contact manifold
\[
(L(q_1,-p_1),\xi_1)\#(M_1=M(\frac{q_2}{p_2},\cdots,\frac{q_n}{p_n}),\zeta_1)
\]
by Legendrian surgery along $K_1$.  Here we are using the fact that if $-p_i/q_i=[a_0^i,a_1^i,\ldots,a_{l_i}^i]$, then
\[
\frac{q_i}{p_i+a_0^iq_i} = [a_1^i,a_2^i,\ldots,a_{l_i}^i].
\]
The contact structures $\xi_1$ and $\zeta_1$ are the obvious ones, obtained from the Stein handlebody diagram in Figure \ref{fig:seifert-filling} by erasing $K_1$.  According to \cite[Theorem 1.3]{menke2018jsj}, every exact symplectic filling of $(M,\xi)$ is obtained from an exact filling of $(L(q_1,-p_1),\xi_1)\#(M_1,\zeta_1)$ by attaching a Weinstein 2-handle along $K_1$.  In the language of round handles, we have Legendrian knots $L_1^-\subset(L(q_1,-p_1),\xi_1)$ and $L_1^+\subset(M_1,\zeta_1)$ along which we may attach a round symplectic 1-handle to a filling of $(L(q_1,-p_1),\xi_1)\sqcup (M_1,\zeta_1)$.\\

We have presented $(M_1,\zeta_1)$ as the boundary of the Stein handlebody depicted in Figure \ref{fig:seifert-filling}, with the chain of knots with framings $a_0^1,a_1^1,\ldots,a_{l_1}^1$ deleted.  By its construction, $(M_1,\zeta_1)$ is lightly mixed, with $K_2,\ldots,K_{n-2}$ having been stabilized both positively and negatively.  We may thus repeat the above procedure to decompose a filling of $(M,\xi)$ into a filling of
\[
(L(q_1,-p_1),\xi_1) \sqcup (L(q_2,-p_2),\xi_2) \sqcup (M_2,\zeta_2).
\]
We continue this procedure until we are left with
\[
(L(q_1,-p_1),\xi_1) \sqcup\cdots\sqcup (L(q_{n-2},-p_{n-2}),\xi_{n-2}) \sqcup (M_{n-2},\zeta_{n-2}),
\]
where $(M_{n-2},\zeta_{n-2})$ is as in Figure \ref{fig:light-diagrams}: there are two horizontal knots, neither of which has stabilizations of both signs.  This is a Seifert fibered space over $S^2$, and thus a lens space.  Indeed, after sliding $K_2$ over $K_1$, we may cancel the 2-handle attached along $K_1$ with the 1-handle.  We are left with a chain of unknots whose framings are given by
\[
a_{l_{n-2}}^{n-2},\ldots,a_1^{n-2},a_0^{n-2}+a_0^{n},a_1^{n},\ldots,a_{l_{n}}^{n},
\]
and thus $M_{n-2}\cong L(p',q')$, where
\[
-\frac{p'}{q'} = [a_{l_{n-2}}^{n-2},\ldots,a_1^{n-2},a_0^{n-2}+a_0^{n},a_1^{n},\ldots,a_{l_{n}}^{n}].
\]
See Figure \ref{fig:light-diagrams}.  This proves Theorem \ref{thm:lightly-mixed}.\\

\begin{figure}
\centering
\begin{tikzpicture}
\begin{scope}[scale=1.25]
\draw (0,0) ellipse (5 and 3);
\draw (0,0) ellipse (4.5 and 2.5);
\draw (0,0) ellipse (4 and 2);
\node[above] at (0,-2.5) {$\infty$};
\node[below] at (0,-3) {$-1$};
\node[above] at (2,-2.3) {$T_1'$};
\node[below] at (2,-2.35) {$T_1$};
\node[above] at (1,-2.5) {$-$};
\node[below] at (1,-2.5) {$+$};
\node[above] at (0,-2) {$e_0+1$};
\end{scope}

\begin{scope}[xshift=-2cm]
\draw (0,0) circle (1.15);
\node[below] at (0,-1.15) {$T_2$};
\node[above] at (0,1.15) {$-1$};
\draw[dashed] (1.15,0) -- (2.85,0);
\node[above] at (2,0) {$A$};
\end{scope}

\begin{scope}[xshift=2cm]
\draw (0,0) circle (1.15);
\node[below] at (0,-1.15) {$T_3$};
\node[above] at (0,1.15) {$-e_0-1$};
\end{scope}
\end{tikzpicture}
\caption{If $M(\frac{q_1}{p_1},\frac{q_2}{p_2},\frac{q_3}{p_3})$ is thoroughly mixed, then $T_1'$ is a mixed torus.}
\label{fig:base-case}
\end{figure}
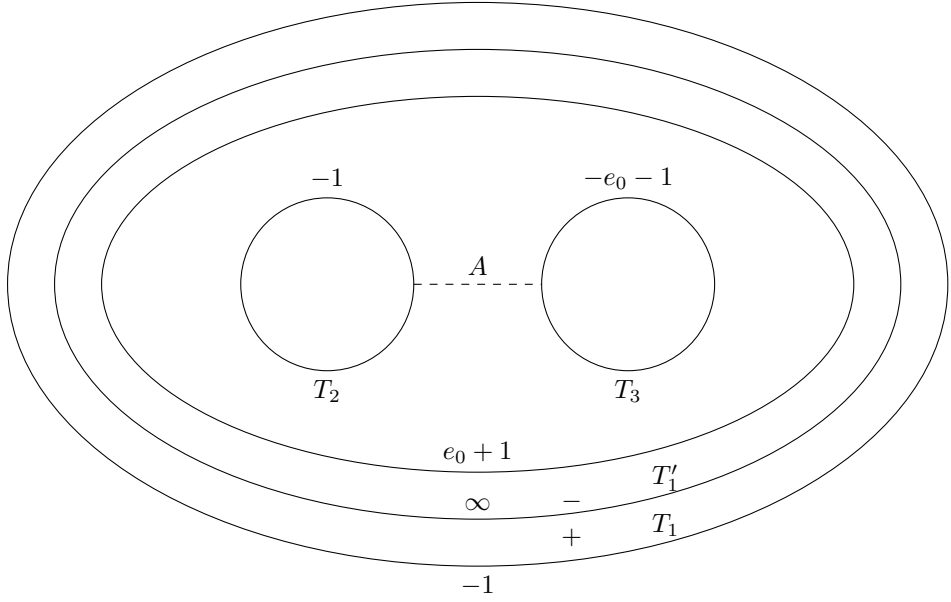

There are some thoroughly mixed contact structures for which $n-1$ of the knots $K_1,\ldots,K_n$ have been stabilized both positively and negatively.  For these, the proof of Theorem \ref{thm:thoroughly-mixed} proceeds as did the proof of Theorem \ref{thm:lightly-mixed}.  But the condition of being thoroughly mixed is more relaxed than this, and we will in fact use Theorem \ref{thm:jsj} directly in our proof, rather than Theorem \ref{thm:menke-knot}.\\

Our argument proceeds by induction on the number $n$ of singular fibers.  Consider first the case where $n=3$.  Then, as depicted in Figure \ref{fig:base-case}, we have a mixed torus $T_1'$ with vertical dividing curves, sandwiched between basic slices whose other tori have dividing curves of slope $-1$ and $e_0+1$, respectively.  Theorem \ref{thm:jsj} would have us split $M=M(\frac{q_1}{p_1},\frac{q_2}{p_2},\frac{q_3}{p_3})$ open along this torus and attach a solid torus to each of the resulting pieces.  Because the dividing curves of $T_1'$ are vertical, the meridian $\mu(S)$ of the solid torus $S$ must have slope $m\in\mathbb{Z}$.  In fact, \cite[Theorem 1.1]{menke2018jsj} tells us that we must have $0\leq m\leq e_0$, since the slopes adjacent to our mixed torus are $-1$ and $e_0+1$.\\

Now one of the two closed contact manifolds is $L_1=S\cup_{T_1'}V_1'$, a gluing of two solid tori.  The meridian of $V_1'$ has slope $q_1/p_1$, with $0<q_1<p_1$.  We may consider a family of tori $T^2\times[0,1]$ in $L_1$ such that $T^2\times\{0\}\subset V_1'$ has dividing curves with slope $q_1/p_1$, $T^2\times\{1/2\}=T_1'$, and $T^2\times\{1\}\subset S$ has dividing curves of slope $m$.  As the dividing curves rotate counterclockwise from $q_1/p_1$ to $\infty$ to $m$, they must not rotate through an angle in excess of $\pi$, since $L_1$ is fillable and thus tight.  This restriction is only satisfied when $m=0$.  So we conclude that $m=0$ and $L_1=L(q_1,-p_1)$.  See Figure \ref{fig:thoroughly-mixed-slopes}.\\

The other closed contact manifold produced by our application of Theorem \ref{thm:jsj} is obtained from $M$ by deleting the neighborhood $V_1'$ of a singular fiber and replacing it with the solid torus $S$, glued in with horizontal meridians.  The result is $M(\frac{0}{1},\frac{q_2}{p_2},\frac{q_3}{p_3})=M(\frac{q_2}{p_2},\frac{q_3}{p_3})$.  We may now apply Theorem \ref{thm:jsj} to this Seifert fibered space (which is in fact a lens space) along the mixed torus $T_2'$.  Arguing as before, we find that the solid torus which is glued in at this stage must have horizontal dividing curves.  The result of this decomposition is a disjoint union of fillings of some contact structures on
\[
L(q_2,-p_2)
\quad\text{and}\quad
M\left(\frac{0}{1},\frac{q_3}{p_3}\right) = L(q_3,-p_3).
\]
Altogether, we have decomposed a filling of $M$ with a thoroughly mixed tight contact structure into a disjoint union of fillings of $L(q_i,-p_i)$, $i=1,2,3$, with some tight contact structures, and Theorem \ref{thm:jsj} provides the Legendrian knots described in Theorem \ref{thm:thoroughly-mixed}.  This establishes the base case of our induction.\\

\begin{figure}
\centering
\begin{tikzpicture}[scale=3]
\draw[thin] (0,-1) -- (0,1);
\draw[thin] (-1,0) -- (1,0);

\draw[thick] (-1,0) -- (1,0);
\node[above] at (-1,0) {$\mu_S$};

\draw[thick] (0,-1) -- (0,1);
\node[right] at (0,.8) {$\Gamma_{T_1'}$};

\draw[thick, dashed] (-0.9191,-0.3939) -- (0.9191,0.3939);
\node[above] at (0.9191,0.3939) {$\mu_{V_1'}$};

\draw [thick,domain=23.2:180] plot ({0.6*cos(\x)}, {0.6*sin(\x)});
\end{tikzpicture}
\caption{Because $0<q_1/p_1<1$, we must have $m=0$.}
\label{fig:thoroughly-mixed-slopes}
\end{figure}
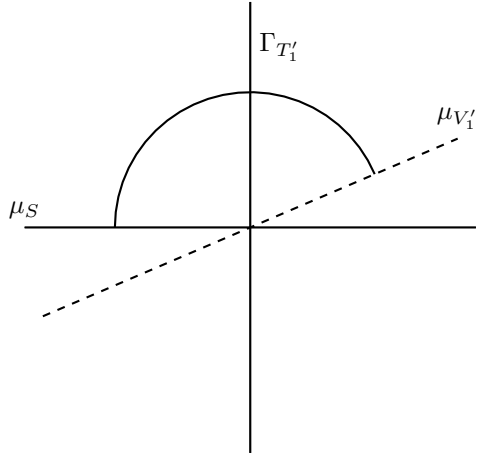

For the inductive step, the analysis above proceeds as before.  Splitting a filling of $M=M(\frac{q_1}{p_1},\ldots,\frac{q_n}{p_n})$ open along the mixed torus $T_1'$ produces symplectic fillings of
\[
L(q_1,-p_1)
\quad\text{and}\quad
M\left(\frac{0}{1},\frac{q_2}{p_2},\ldots,\frac{q_n}{p_n}\right).
\]
The latter is a thoroughly mixed Seifert fibered space with $n-1$ singular fibers, for which we assume that Theorem \ref{thm:thoroughly-mixed} holds, and thus the decomposition may continue until we have a disjoint union of filling of $L(q_i,-p_i)$, for $i=1,\ldots,n$.  This proves Theorem \ref{thm:thoroughly-mixed}.\\

\begin{figure}
\centering
\begin{tikzpicture}[xscale=0.5,yscale=0.6]
\begin{scope}
\draw (-6,0) ellipse (2 and 4);
\draw (-8,0) arc (180:360:2 and 0.3);
\draw[dashed] (-4,0) arc (0:180:2 and 0.3);
\draw (6,0) ellipse (2 and 4);
\draw (4,0) arc (180:360:2 and 0.3);
\draw[dashed] (8,0) arc (0:180:2 and 0.3);

\begin{knot}[
	clip width=5,
	clip radius=2pt,
	ignore endpoint intersections=false,
 	flip crossing/.list={2,4,6,8,10,12}]
\strand (-4.72,3) to[out=0,in=180]
	    (3,2.5) to[out=180,in=0]
	    (2,3) to[out=0,in=180]
	    (4.72,3);
\strand (-3.5,2.375) to[out=0,in=180]
		(-2.5,3) to[out=0,in=180]
		(-1.5,2.375) to[out=180,in=0]
		(-2,2) to[out=0,in=180]
		(-1.5,1.625) to[out=180,in=0]
		(-2.5,1.25) to[out=180,in=0]
		(-3.5,1.625) to[out=0,in=180]
		(-3,2) to[out=180,in=0]
		(-3.5,2.375);
\strand (-2.25,1.625) to[out=0,in=180]
		(-0.75,2.125) to[out=0,in=180]
		(0,2) to[out=180,in=0]
		(-0.5,1.625) to[out=0,in=180]
		(0,1.25) to[out=180,in=0]
		(-0.75,1.125) to[out=180,in=0]
		(-2.25,1.625);
\strand (-0.75,2) to[out=0,in=180]
		(0.75,2.5) to[out=0,in=180]
		(1.5,2.375) to[out=180,in=0]
		(1,2) to[out=0,in=180]
		(1.5,1.625) to[out=180,in=0]
		(0.75,1.5) to[out=180,in=0]
		(-0.75,2);
	    
\strand (-4,0) to[out=0,in=180]
	    (3,-0.5) to[out=180,in=0]
	    (2,0) to[out=0,in=180]
	    (4,0);
	
\strand (-4.72,-3.125) to[out=0,in=180]
	    (-2,-3.125) to[out=180,in=0]
	    (-3,-3.625) to[out=0,in=180]
	    (4.72,-3.125);
\strand (-0.75,-3.625) to[out=0,in=180]
		(0,-3.375) to[out=0,in=180]
		(1.5,-4.25) to[out=180,in=0]
		(0,-5.125) to[out=180,in=0]
		(-0.75,-4.875) to[out=0,in=180]
		(-0.25,-4.5625) to[out=180,in=0]
		(-0.75,-4.25) to[out=0,in=180]
		(-0.25,-3.9375) to[out=180,in=0]
		(-0.75,-3.625);
\strand (0.875,-4.25) to[out=0,in=180]
		(1.625,-4.125) to[out=0,in=180]
		(3.25,-4.625) to[out=180,in=0]
		(1.625,-5.125) to[out=180,in=0]
		(0.875,-5) to[out=0,in=180]
		(1.375,-4.625) to[out=180,in=0]
		(0.875,-4.25);
\strand (2.125,-4.625) to[out=0,in=180]
		(3.75,-4.125) to[out=0,in=180]
		(4.5,-4.25) to[out=180,in=0]
		(4,-4.625) to[out=0,in=180]
		(4.5,-5) to[out=180,in=0]
		(3.75,-5.125) to[out=180,in=0]
		(2.125,-4.625);
\end{knot}

\node[above] at (4,3) {$K_1$};
\node[above] at (3.25,0) {$K_2$};
\begin{scope}[yshift=-6cm]
\node[above] at (3.75,3) {$K_3$};
\end{scope}
\end{scope}

\draw[->] (2,-5.5) -- (6,-8.5);
\draw[->] (-2,-5.5) -- (-6,-8.5);

\begin{scope}[xshift=-9cm,yshift=-12cm]
\draw (-6,0) ellipse (2 and 4);
\draw (-8,0) arc (180:360:2 and 0.3);
\draw[dashed] (-4,0) arc (0:180:2 and 0.3);
\draw (6,0) ellipse (2 and 4);
\draw (4,0) arc (180:360:2 and 0.3);
\draw[dashed] (8,0) arc (0:180:2 and 0.3);

\begin{knot}[
	clip width=5,
	clip radius=2pt,
	ignore endpoint intersections=false,
 	flip crossing/.list={2,4,6,8,10,12}]
\strand (-4.72,3) to[out=0,in=180]
	    (3,2.5) to[out=180,in=0]
	    (2,3) to[out=0,in=180]
	    (4.72,3);
\strand[dashed] (-3.5,2.375) to[out=0,in=180]
		(-2.5,3) to[out=0,in=180]
		(-1.5,2.375) to[out=180,in=0]
		(-2,2) to[out=0,in=180]
		(-1.5,1.625) to[out=180,in=0]
		(-2.5,1.25) to[out=180,in=0]
		(-3.5,1.625) to[out=0,in=180]
		(-3,2) to[out=180,in=0]
		(-3.5,2.375);
\strand (-2.25,1.625) to[out=0,in=180]
		(-0.75,2.125) to[out=0,in=180]
		(0,2) to[out=180,in=0]
		(-0.5,1.625) to[out=0,in=180]
		(0,1.25) to[out=180,in=0]
		(-0.75,1.125) to[out=180,in=0]
		(-2.25,1.625);
\strand (-0.75,2) to[out=0,in=180]
		(0.75,2.5) to[out=0,in=180]
		(1.5,2.375) to[out=180,in=0]
		(1,2) to[out=0,in=180]
		(1.5,1.625) to[out=180,in=0]
		(0.75,1.5) to[out=180,in=0]
		(-0.75,2);
	    
\strand (-4,0) to[out=0,in=180]
	    (3,-0.5) to[out=180,in=0]
	    (2,0) to[out=0,in=180]
	    (4,0);
	
\strand (-4.72,-3.125) to[out=0,in=180]
	    (-2,-3.125) to[out=180,in=0]
	    (-3,-3.625) to[out=0,in=180]
	    (4.72,-3.125);
\strand (-0.75,-3.625) to[out=0,in=180]
		(0,-3.375) to[out=0,in=180]
		(1.5,-4.25) to[out=180,in=0]
		(0,-5.125) to[out=180,in=0]
		(-0.75,-4.875) to[out=0,in=180]
		(-0.25,-4.5625) to[out=180,in=0]
		(-0.75,-4.25) to[out=0,in=180]
		(-0.25,-3.9375) to[out=180,in=0]
		(-0.75,-3.625);
\strand[dashed] (0.875,-4.25) to[out=0,in=180]
		(1.625,-4.125) to[out=0,in=180]
		(3.25,-4.625) to[out=180,in=0]
		(1.625,-5.125) to[out=180,in=0]
		(0.875,-5) to[out=0,in=180]
		(1.375,-4.625) to[out=180,in=0]
		(0.875,-4.25);
\strand (2.125,-4.625) to[out=0,in=180]
		(3.75,-4.125) to[out=0,in=180]
		(4.5,-4.25) to[out=180,in=0]
		(4,-4.625) to[out=0,in=180]
		(4.5,-5) to[out=180,in=0]
		(3.75,-5.125) to[out=180,in=0]
		(2.125,-4.625);
\end{knot}

\node[above] at (4,3) {$K_1$};
\node[above] at (3.25,0) {$K_2$};
\begin{scope}[yshift=-6cm]
\node[above] at (3.75,3) {$K_3$};
\end{scope}
\end{scope}

\begin{scope}[xshift=9cm,yshift=-12cm]
\draw (-6,0) ellipse (2 and 4);
\draw (-8,0) arc (180:360:2 and 0.3);
\draw[dashed] (-4,0) arc (0:180:2 and 0.3);
\draw (6,0) ellipse (2 and 4);
\draw (4,0) arc (180:360:2 and 0.3);
\draw[dashed] (8,0) arc (0:180:2 and 0.3);

\begin{knot}[
	clip width=5,
	clip radius=2pt,
	ignore endpoint intersections=false,
 	flip crossing/.list={2,4,6,8,10,12}]
\strand (-4.72,3) to[out=0,in=180]
	    (3,2.5) to[out=180,in=0]
	    (2,3) to[out=0,in=180]
	    (4.72,3);
\strand[dashed] (-3.5,2.375) to[out=0,in=180]
		(-2.5,3) to[out=0,in=180]
		(-1.5,2.375) to[out=180,in=0]
		(-2,2) to[out=0,in=180]
		(-1.5,1.625) to[out=180,in=0]
		(-2.5,1.25) to[out=180,in=0]
		(-3.5,1.625) to[out=0,in=180]
		(-3,2) to[out=180,in=0]
		(-3.5,2.375);
\strand (-2.25,1.625) to[out=0,in=180]
		(-0.75,2.125) to[out=0,in=180]
		(0,2) to[out=180,in=0]
		(-0.5,1.625) to[out=0,in=180]
		(0,1.25) to[out=180,in=0]
		(-0.75,1.125) to[out=180,in=0]
		(-2.25,1.625);
\strand (-0.75,2) to[out=0,in=180]
		(0.75,2.5) to[out=0,in=180]
		(1.5,2.375) to[out=180,in=0]
		(1,2) to[out=0,in=180]
		(1.5,1.625) to[out=180,in=0]
		(0.75,1.5) to[out=180,in=0]
		(-0.75,2);
	    
\strand (-4,0) to[out=0,in=180]
	    (3,-0.5) to[out=180,in=0]
	    (2,0) to[out=0,in=180]
	    (4,0);
	
\strand (-4.72,-3.125) to[out=0,in=180]
	    (-2,-3.125) to[out=180,in=0]
	    (-3,-3.625) to[out=0,in=180]
	    (4.72,-3.125);
\strand (-0.75,-3.625) to[out=0,in=180]
		(0,-3.375) to[out=0,in=180]
		(1.5,-4.25) to[out=180,in=0]
		(0,-5.125) to[out=180,in=0]
		(-0.75,-4.875) to[out=0,in=180]
		(-0.25,-4.5625) to[out=180,in=0]
		(-0.75,-4.25) to[out=0,in=180]
		(-0.25,-3.9375) to[out=180,in=0]
		(-0.75,-3.625);
\strand (0.875,-4.25) to[out=0,in=180]
		(1.625,-4.125) to[out=0,in=180]
		(3.25,-4.625) to[out=180,in=0]
		(1.625,-5.125) to[out=180,in=0]
		(0.875,-5) to[out=0,in=180]
		(1.375,-4.625) to[out=180,in=0]
		(0.875,-4.25);
\strand[dashed] (2.125,-4.625) to[out=0,in=180]
		(3.75,-4.125) to[out=0,in=180]
		(4.5,-4.25) to[out=180,in=0]
		(4,-4.625) to[out=0,in=180]
		(4.5,-5) to[out=180,in=0]
		(3.75,-5.125) to[out=180,in=0]
		(2.125,-4.625);
\end{knot}

\node[above] at (4,3) {$K_1$};
\node[above] at (3.25,0) {$K_2$};
\begin{scope}[yshift=-6cm]
\node[above] at (3.75,3) {$K_3$};
\end{scope}
\end{scope}
\end{tikzpicture}
\caption{Decomposing a filling of a contact structure which is neither thoroughly nor lightly mixed.  The result is a filling of a disjoint union of a universally tight small Seifert fibered space and some universally tight lens spaces.}
\label{fig:ssf-example}
\end{figure}
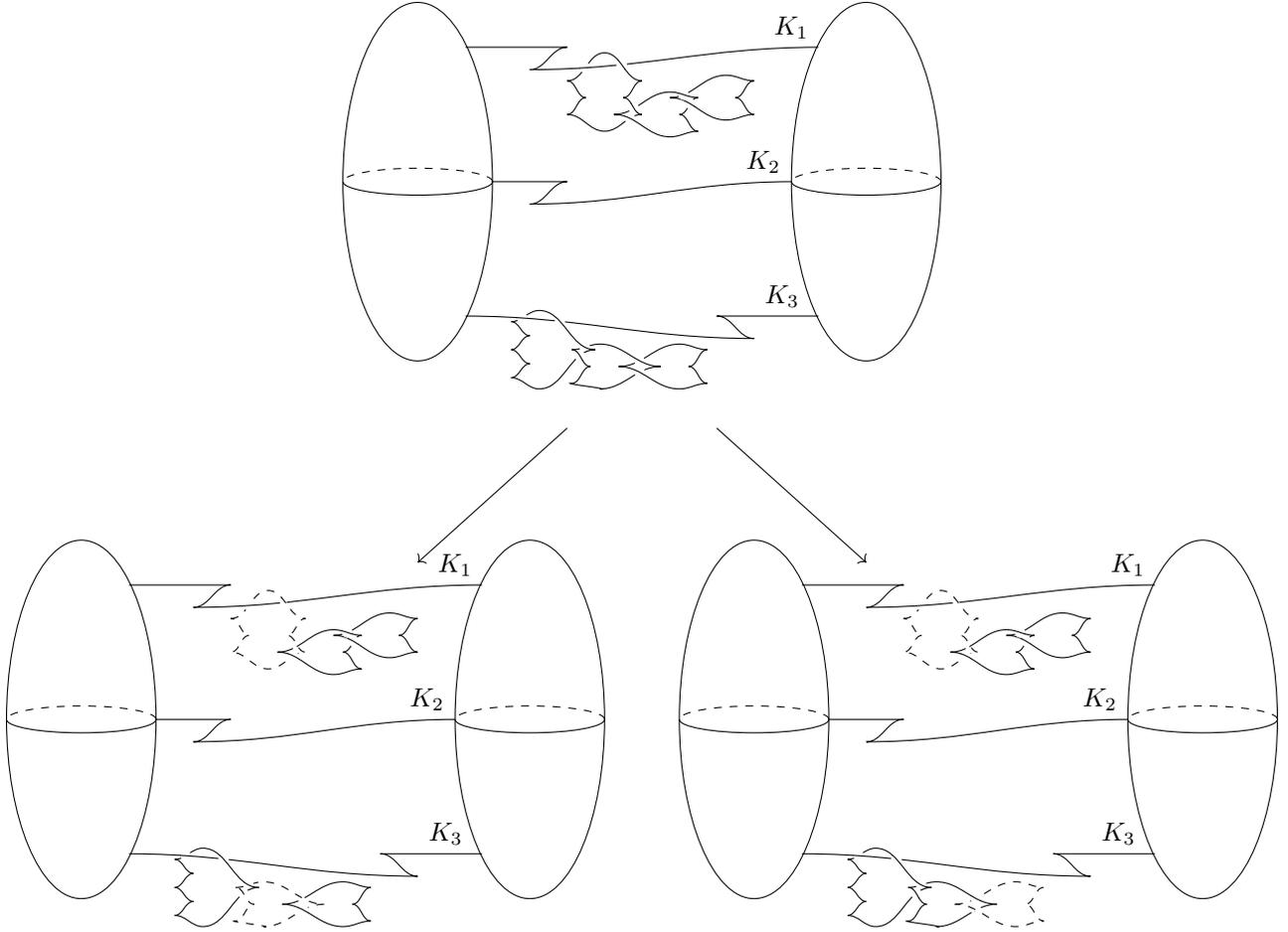

At last, we address fillings of those contact structures on small Seifert fibered spaces which have at least one horizontal link with both positive and negative stabilizations --- these are the structures considered in Lemma \ref{lemma:vot}.  In this case, each of $K_1,K_2$, and $K_3'$ has stabilizations of a single sign, but these signs do not all agree.  Here, as above, $K_3'$ is the nearest unknot adjacent to $K_3$ which has been stabilized, meaning that $K_3'=K_3$ if $e_0=0$.  For $i=1,2,3$, we let $\widehat{K}_i$ denote the nearest knot adjacent to $K_i$ with a stabilization of a different sign from those on $K_i$ (or $K_3'$).  By our assumption, at least one $\widehat{K}_i$ exists.  Let us write
\begin{equation}\label{eq:ssf}
M = M\left(\frac{q_1}{p_1},\frac{q_2}{p_2},\frac{q_3}{p_3}\right)\cong(\Sigma\times S^1) \cup \left(\bigcup_{i=1}^3\bigcup_{j=1}^{l_i}L_{i,j}\right)\cup\left(\bigcup_{i=1}^3 V_i\right),
\end{equation}
where each $V_i$ is a solid torus, $-\partial(\Sigma\times S^1)=T_1+T_2+T_3$, and each $L_{i,j}\cong T^2\times I$ is a continued fraction block corresponding to a knot in the surgery diagram for $(M,\xi)$.  Specifically, let $L_{i,j_i}$ be the continued fraction block corresponding to $\widehat{K}_i$, for $i=1,2,3$.  Then the boundary torus $\widehat{T}_i$ between $L_{i,j_i}$ and $L_{i,j_i-1}$ is a mixed torus, and each continued fraction block preceding $L_{i,j_i}$ has basic slices of a single sign, matching the stabilizations of $K_i$.\\

Notice that simultaneously splitting $(M,\xi)$ along the mixed tori $\widehat{T}_1,\widehat{T}_2,$ and $\widehat{T}_3$ yields a disjoint union of a universally tight small Seifert fibered space and three lens spaces, independent of the slopes which are used to perform this splitting.  It follows that by applying the JSJ decomposition to an exact symplectic filling of $(M,\xi)$, we may obtain this filling from a disjoint union of an exact filling of a universally tight small Seifert fibered space with exact fillings of three lens spaces.  By applying Theorem \ref{thm:lens-space-fillings} to the three lens space fillings, we prove Theorem \ref{thm:not-mixed}.\\

Observe that for contact structures which are thoroughly or lightly mixed, the conclusion of Theorem \ref{thm:not-mixed} follows from Theorems \ref{thm:lens-space-fillings}, \ref{thm:thoroughly-mixed}, and \ref{thm:lightly-mixed}.  So, with the small number of exceptions pointed out at the conclusion of Section~\ref{subsubsec:mixed-structures}, we have reduced the problem of classifying exact symplectic fillings for small Seifert fibered spaces to the same problem for universally tight lens spaces and for universally tight small Seifert fibered spaces.  See Figure \ref{fig:ssf-example}.

\subsubsection{The case $e_0\leq -3$}\label{sec:negative-euler}
Throughout this section, we will consider a Seifert fibered space $M=M(-\frac{q_1}{p_1},\cdots,-\frac{q_n}{p_n})$ with $p_i\geq 2$, $q_i\geq 1$, and $(p_i,q_i)=1$ for $i=1,\ldots,n$.  Every tight contact structure $\xi$ on $M$ that we consider will be constructed by putting the knots of Figure \ref{fig:negative-euler} into Legendrian position and stabilizing appropriately.\\

In case $\xi$ is centrally mixed --- meaning that the central knot of Figure \ref{fig:negative-euler} is stabilized both positively and negatively --- Proposition \ref{prop:centrally-mixed} tells us that the exact symplectic fillings of $(M,\xi)$ are obtained by attaching a sequence of round symplectic 1-handles to a disjoint union of fillings of lens spaces.  If $\xi$ is not centrally mixed, then the central knot has stabilizations which are either all positive or all negative; notice that, since $e_0\leq -n\leq -3$, the central knot must have at least one stabilization.  The following proposition considers the case in which the stabilizations of the central knot are all of a single sign.

\begin{proposition}\label{prop:legs-have-single-sign}
Let $(M,\xi)$ be as above, with the central knot in Figure \ref{fig:negative-euler} having stabilizations which are all of a single sign.  Then every exact symplectic filling of $(M,\xi)$ may be obtained by attaching round symplectic 1-handles to a disjoint union of fillings of lens spaces and a Seifert fibered space $(M',\xi')$.  Moreover, $(M',\xi')$ admits a Legendrian surgery diagram as in Figure \ref{fig:negative-euler}, with each leg of the diagram having stabilizations of a single sign.
\end{proposition}
\begin{proof}
If the Legendrian surgery diagram for $(M,\xi)$ is such that no leg has both positive and negative stabilizations, then we have nothing to do --- $(M',\xi')$ is simply $(M,\xi)$.  Otherwise, we lose no generality by assuming that the first leg of Figure \ref{fig:negative-euler} has both positive and negative stabilizations.  We will reduce to a case where the first leg does not have both positive and negative stabilizations; by applying this argument to each leg of the diagram, we obtain the desired result.\\

Note that if the first leg of our diagram contains a knot which is stabilized both positively and negatively, then we may apply Theorem \ref{thm:menke-knot} to this knot, amputating from the diagram this knot and all those below it in the leg.  Thus we assume that each knot in the first leg of our diagram has stabilizations of a single sign.  We then have knots $K^1_i$ and $K^1_j$, $1\leq i < j \leq \ell_1$, which have stabilizations of opposite signs, and are such that any knot $K^1_k$, with $i < k < j$, has no stabilizations.  Let us assume that $i$ and $j$ are minimal among such indices.  (Here we are using the notation established in Figure \ref{fig:negative-euler}.)  We will identify a mixed torus in $(M,\xi)$ associated to this mismatch of signs.\\

To this end, we decompose $(M,\xi)$ as
\[
M \cong (\Sigma\times S^1) \cup_{(\varphi_1\cup\cdots\cup\varphi_n)} (V_1\cup\cdots\cup V_n),
\]
where
\begin{itemize}
	\item $\Sigma$ is a planar surface such that $-\partial(\Sigma\times S^1)=T_1+\cdots+T_n$;
	\item each $T_i$ is a minimal convex torus with dividing curves of slope $\lfloor\frac{q_i}{p_i}\rfloor$;
	\item each $V_i$ is a solid torus, and $\partial V_i$ has dividing curves of slope $-\dfrac{q_i-\lfloor\frac{q_i}{p_i}\rfloor p_i}{q'_i-\lfloor\frac{q_i}{p_i}\rfloor p'_i}$, where $p_i\geq p'_i>0$, $q_i\geq q'_i>0$, and $p_iq'_i-q_ip'_i=1$;
	\item the gluing maps $\varphi_i\colon\partial V_i\to T_i$ are defined by
	\[
	\varphi_i = \left(\begin{matrix}
	p_i & p'_i\\ q_i & q'_i
	\end{matrix}\right).
	\]
\end{itemize}
The solid torus $(V_1,\xi|_{V_1})$ may be further decomposed by peeling off basic slices until we are left with a solid torus whose boundary has dividing curves of slope $-1$.  In this decomposition, we have a continued fraction block of basic slices for each knot $K^1_1,\ldots,K^1_{\ell_1}$ which has been stabilized.  In particular, the knots $K^1_i$ and $K^1_j$ correspond to adjacent continued fraction blocks.  By assumption, these continued fraction blocks are universally tight and of opposite sign, and thus their common boundary $T$ is a mixed torus.  We may normalize the slope of the dividing curves on $T$ to be $\infty$, and on the opposite boundary of the negative basic slice to be $-1$.  As in previous iterations of this argument, the slope $s$ on the opposite boundary of the positive basic slice will depend on the number of unstabilized knots which exist between $K^1_i$ and $K^1_j$.  In particular, $s$ will be two more than the number of these knots.\\

We now apply Theorem \ref{thm:jsj} to a filling of $(M,\xi)$ along $T$.  There are $s$ contact manifolds which might be produced by this decomposition, and these correspond diagrammatically to deleting either $K^1_i$, $K^1_j$, or one of the intermediate knots from Figure \ref{fig:negative-euler}.  In any case, the contact manifold is a disjoint union of a lens space (whose surgery diagram is given by the link below the deleted knot) and a Seifert fibered space $(M',\xi')$.  The stabilizations in the first leg of the Legendrian surgery diagram for $(M',\xi')$ --- of which there may be none --- all have the same sign as those of $K^1_i$.  In particular, the first leg does not have both positive and negative stabilizations.  By applying this theorem to each leg with both positive and negative stabilizations, we reduce to the case of lens spaces and Seifert fibered spaces each of whose legs has stabilizations of a single sign.
\end{proof}

\begin{figure}
\centering
\begin{tikzpicture}[scale=0.5]
\begin{knot}[
	clip width=5,
	clip radius=2pt,
	ignore endpoint intersections=false,
	flip crossing/.list={2,4,6,7,9,11,13}]
\strand (0,0) to[out=0,in=180]
        (3,0.5) to[out=0,in=180]
        (9,-1) to[out=180,in=0]
        (3,-2.5) to[out=180,in=0]
        (0,-2) to[out=0,in=180]
        (1,-1.5) to[out=180,in=0]
        (0,-1) to[out=0,in=180]
        (1,-0.5) to[out=180,in=0]
        (0,0);

\strand (0,-2.5) to[out=0,in=180]
		(1,-2) to[out=0,in=180]
		(2,-3) to[out=180,in=0]
		(1,-4) to[out=180,in=0]
		(0,-3.5) to[out=0,in=180]
		(0.5,-3) to[out=180,in=0]
		(0,-2.5);
\strand (0,-4) to[out=0,in=180]
		(1,-3.5) to[out=0,in=180]
		(2,-4.5) to[out=180,in=0]
		(1,-5.5) to[out=180,in=0]
		(0,-5) to[out=0,in=180]
		(0.5,-4.5) to[out=180,in=0]
		(0,-4);

\strand (3,-2.7) to[out=0,in=180]
		(4,-2.3) to[out=0,in=180]
		(5,-3.5) to[out=180,in=0]
		(4,-4.7) to[out=180,in=0]
		(3,-4.3) to[out=0,in=180]
		(3.5,-3.9) to[out=180,in=0]
		(3,-3.5) to[out=0,in=180]
		(3.5,-3.1) to[out=180,in=0]
		(3,-2.7);
\strand (3,-4.75) to[out=0,in=180]
		(4,-4.25) to[out=0,in=180]
		(5,-5.25) to[out=180,in=0]
		(4,-6.25) to[out=180,in=0]
		(3,-5.75) to[out=0,in=180]
		(3.5,-5.25) to[out=180,in=0]
		(3,-4.75);

\strand (6,-2) to[out=0,in=180]
		(7,-1.3) to[out=0,in=180]
		(8,-2) to[out=180,in=0]
		(7,-2.7) to[out=180,in=0]
		(6,-2);
\strand (6,-3.5) to[out=0,in=180]
		(7,-2.3) to[out=0,in=180]
		(8,-2.7) to[out=180,in=0]
		(7.5,-3.1) to[out=0,in=180]
		(8,-3.5) to[out=180,in=0]
		(7.5,-3.9) to[out=0,in=180]
		(8,-4.3) to[out=180,in=0]
		(7,-4.7) to[out=180,in=0]
		(6,-3.5);
\strand (6,-5.25) to[out=0,in=180]
		(7,-4.25) to[out=0,in=180]
		(8,-4.75) to[out=180,in=0]
		(7.5,-5.25) to[out=0,in=180]
		(8,-5.75) to[out=180,in=0]
		(7,-6.25) to[out=180,in=0]
		(6,-5.25);
\end{knot}

\draw[->] (0,-7) -- (-2,-9);
\draw[->] (4.5,-7) -- (4.5,-9);
\draw[->] (9,-7) -- (11,-9);

\begin{scope}[xshift=-11cm,yshift=-10cm]
\begin{knot}[
	clip width=5,
	clip radius=2pt,
	ignore endpoint intersections=false,
	flip crossing/.list={2,4,6,7,9,11,13}]
\strand[dashed] (0,0) to[out=0,in=180]
        (3,0.5) to[out=0,in=180]
        (9,-1) to[out=180,in=0]
        (3,-2.5) to[out=180,in=0]
        (0,-2) to[out=0,in=180]
        (1,-1.5) to[out=180,in=0]
        (0,-1) to[out=0,in=180]
        (1,-0.5) to[out=180,in=0]
        (0,0);

\strand (0,-2.5) to[out=0,in=180]
		(1,-2) to[out=0,in=180]
		(2,-3) to[out=180,in=0]
		(1,-4) to[out=180,in=0]
		(0,-3.5) to[out=0,in=180]
		(0.5,-3) to[out=180,in=0]
		(0,-2.5);
\strand (0,-4) to[out=0,in=180]
		(1,-3.5) to[out=0,in=180]
		(2,-4.5) to[out=180,in=0]
		(1,-5.5) to[out=180,in=0]
		(0,-5) to[out=0,in=180]
		(0.5,-4.5) to[out=180,in=0]
		(0,-4);

\strand (3,-2.7) to[out=0,in=180]
		(4,-2.3) to[out=0,in=180]
		(5,-3.5) to[out=180,in=0]
		(4,-4.7) to[out=180,in=0]
		(3,-4.3) to[out=0,in=180]
		(3.5,-3.9) to[out=180,in=0]
		(3,-3.5) to[out=0,in=180]
		(3.5,-3.1) to[out=180,in=0]
		(3,-2.7);
\strand (3,-4.75) to[out=0,in=180]
		(4,-4.25) to[out=0,in=180]
		(5,-5.25) to[out=180,in=0]
		(4,-6.25) to[out=180,in=0]
		(3,-5.75) to[out=0,in=180]
		(3.5,-5.25) to[out=180,in=0]
		(3,-4.75);

\strand (6,-2) to[out=0,in=180]
		(7,-1.3) to[out=0,in=180]
		(8,-2) to[out=180,in=0]
		(7,-2.7) to[out=180,in=0]
		(6,-2);
\strand (6,-3.5) to[out=0,in=180]
		(7,-2.3) to[out=0,in=180]
		(8,-2.7) to[out=180,in=0]
		(7.5,-3.1) to[out=0,in=180]
		(8,-3.5) to[out=180,in=0]
		(7.5,-3.9) to[out=0,in=180]
		(8,-4.3) to[out=180,in=0]
		(7,-4.7) to[out=180,in=0]
		(6,-3.5);
\strand (6,-5.25) to[out=0,in=180]
		(7,-4.25) to[out=0,in=180]
		(8,-4.75) to[out=180,in=0]
		(7.5,-5.25) to[out=0,in=180]
		(8,-5.75) to[out=180,in=0]
		(7,-6.25) to[out=180,in=0]
		(6,-5.25);
\end{knot}
\end{scope}

\begin{scope}[yshift=-10cm]
\begin{knot}[
	clip width=5,
	clip radius=2pt,
	ignore endpoint intersections=false,
	flip crossing/.list={2,4,6,7,9,11,13}]
\strand (0,0) to[out=0,in=180]
        (3,0.5) to[out=0,in=180]
        (9,-1) to[out=180,in=0]
        (3,-2.5) to[out=180,in=0]
        (0,-2) to[out=0,in=180]
        (1,-1.5) to[out=180,in=0]
        (0,-1) to[out=0,in=180]
        (1,-0.5) to[out=180,in=0]
        (0,0);

\strand (0,-2.5) to[out=0,in=180]
		(1,-2) to[out=0,in=180]
		(2,-3) to[out=180,in=0]
		(1,-4) to[out=180,in=0]
		(0,-3.5) to[out=0,in=180]
		(0.5,-3) to[out=180,in=0]
		(0,-2.5);
\strand (0,-4) to[out=0,in=180]
		(1,-3.5) to[out=0,in=180]
		(2,-4.5) to[out=180,in=0]
		(1,-5.5) to[out=180,in=0]
		(0,-5) to[out=0,in=180]
		(0.5,-4.5) to[out=180,in=0]
		(0,-4);

\strand (3,-2.7) to[out=0,in=180]
		(4,-2.3) to[out=0,in=180]
		(5,-3.5) to[out=180,in=0]
		(4,-4.7) to[out=180,in=0]
		(3,-4.3) to[out=0,in=180]
		(3.5,-3.9) to[out=180,in=0]
		(3,-3.5) to[out=0,in=180]
		(3.5,-3.1) to[out=180,in=0]
		(3,-2.7);
\strand (3,-4.75) to[out=0,in=180]
		(4,-4.25) to[out=0,in=180]
		(5,-5.25) to[out=180,in=0]
		(4,-6.25) to[out=180,in=0]
		(3,-5.75) to[out=0,in=180]
		(3.5,-5.25) to[out=180,in=0]
		(3,-4.75);

\strand[dashed] (6,-2) to[out=0,in=180]
		(7,-1.3) to[out=0,in=180]
		(8,-2) to[out=180,in=0]
		(7,-2.7) to[out=180,in=0]
		(6,-2);
\strand (6,-3.5) to[out=0,in=180]
		(7,-2.3) to[out=0,in=180]
		(8,-2.7) to[out=180,in=0]
		(7.5,-3.1) to[out=0,in=180]
		(8,-3.5) to[out=180,in=0]
		(7.5,-3.9) to[out=0,in=180]
		(8,-4.3) to[out=180,in=0]
		(7,-4.7) to[out=180,in=0]
		(6,-3.5);
\strand (6,-5.25) to[out=0,in=180]
		(7,-4.25) to[out=0,in=180]
		(8,-4.75) to[out=180,in=0]
		(7.5,-5.25) to[out=0,in=180]
		(8,-5.75) to[out=180,in=0]
		(7,-6.25) to[out=180,in=0]
		(6,-5.25);
\end{knot}
\end{scope}

\begin{scope}[xshift=11cm,yshift=-10cm]
\begin{knot}[
	clip width=5,
	clip radius=2pt,
	ignore endpoint intersections=false,
	flip crossing/.list={2,4,6,7,9,11,13}]
\strand (0,0) to[out=0,in=180]
        (3,0.5) to[out=0,in=180]
        (9,-1) to[out=180,in=0]
        (3,-2.5) to[out=180,in=0]
        (0,-2) to[out=0,in=180]
        (1,-1.5) to[out=180,in=0]
        (0,-1) to[out=0,in=180]
        (1,-0.5) to[out=180,in=0]
        (0,0);

\strand (0,-2.5) to[out=0,in=180]
		(1,-2) to[out=0,in=180]
		(2,-3) to[out=180,in=0]
		(1,-4) to[out=180,in=0]
		(0,-3.5) to[out=0,in=180]
		(0.5,-3) to[out=180,in=0]
		(0,-2.5);
\strand (0,-4) to[out=0,in=180]
		(1,-3.5) to[out=0,in=180]
		(2,-4.5) to[out=180,in=0]
		(1,-5.5) to[out=180,in=0]
		(0,-5) to[out=0,in=180]
		(0.5,-4.5) to[out=180,in=0]
		(0,-4);

\strand (3,-2.7) to[out=0,in=180]
		(4,-2.3) to[out=0,in=180]
		(5,-3.5) to[out=180,in=0]
		(4,-4.7) to[out=180,in=0]
		(3,-4.3) to[out=0,in=180]
		(3.5,-3.9) to[out=180,in=0]
		(3,-3.5) to[out=0,in=180]
		(3.5,-3.1) to[out=180,in=0]
		(3,-2.7);
\strand (3,-4.75) to[out=0,in=180]
		(4,-4.25) to[out=0,in=180]
		(5,-5.25) to[out=180,in=0]
		(4,-6.25) to[out=180,in=0]
		(3,-5.75) to[out=0,in=180]
		(3.5,-5.25) to[out=180,in=0]
		(3,-4.75);

\strand (6,-2) to[out=0,in=180]
		(7,-1.3) to[out=0,in=180]
		(8,-2) to[out=180,in=0]
		(7,-2.7) to[out=180,in=0]
		(6,-2);
\strand[dashed] (6,-3.5) to[out=0,in=180]
		(7,-2.3) to[out=0,in=180]
		(8,-2.7) to[out=180,in=0]
		(7.5,-3.1) to[out=0,in=180]
		(8,-3.5) to[out=180,in=0]
		(7.5,-3.9) to[out=0,in=180]
		(8,-4.3) to[out=180,in=0]
		(7,-4.7) to[out=180,in=0]
		(6,-3.5);
\strand (6,-5.25) to[out=0,in=180]
		(7,-4.25) to[out=0,in=180]
		(8,-4.75) to[out=180,in=0]
		(7.5,-5.25) to[out=0,in=180]
		(8,-5.75) to[out=180,in=0]
		(7,-6.25) to[out=180,in=0]
		(6,-5.25);
\end{knot}
\end{scope}
\end{tikzpicture}
\caption{In this example, there are three contact manifolds which might result from the JSJ decomposition for symplectic fillings.  Notice that each consists of universally tight lens spaces, and possibly a Seifert fibered space with canonical contact structure.}
\label{fig:reduce-to-canonical}
\end{figure}
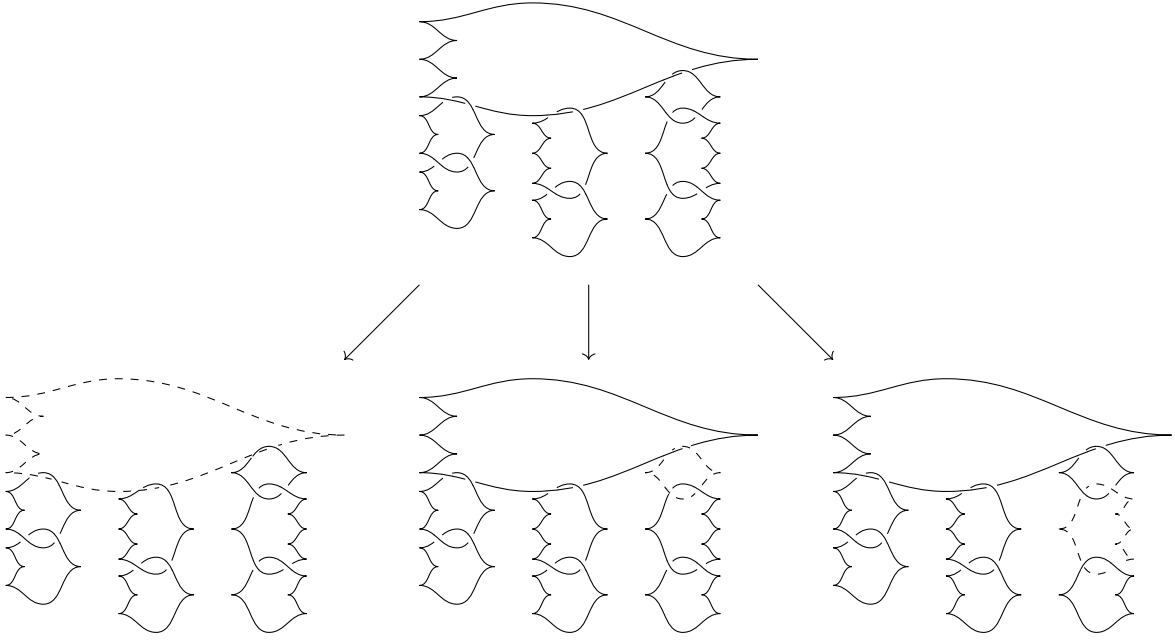

Proposition~\ref{prop:legs-have-single-sign} leads us to consider the case where $(M,\xi)$ is not centrally mixed, and each of its legs has stabilizations of a single sign --- though these signs may not all agree.

\begin{proposition}\label{prop:single-sign}
Let $(M,\xi)$ be as above, with the central knot in Figure \ref{fig:negative-euler} having stabilizations which are all of a single sign, and with each leg having stabilizations of a single sign.  Then every exact symplectic filling of $(M,\xi)$ may be obtained by attaching round symplectic 1-handles to a disjoint union of fillings of lens spaces and a canonical Seifert fibered space.
\end{proposition}
\begin{proof}
We lose no generality by assuming that the stabilizations of the central knot of $(M,\xi)$ are all positive --- as noted above, the central knot must have at least one stabilization.  If the $i$th leg of the Legendrian surgery diagram for $(M,\xi)$ has negative stabilizations, we will identify a mixed torus which allows us to amputate this leg.  Now consider decomposing $(M,\xi)$ as
\[
M \cong (\Sigma\times S^1)\cup_{(\varphi_1\cup\cdots\cup\varphi_n)}(V_1\cup\cdots\cup V_n),
\]
as in the proof of Proposition \ref{prop:legs-have-single-sign}.  Namely, $-\partial(\Sigma\times S^1)=T_1+\cdots+T_n$, where each $T_i$ is a minimal convex torus with dividing curves of slope $\lfloor\frac{q_i}{p_i}\rfloor$.  We claim that there is a positive basic slice adjacent to $T_i$ in $\Sigma\times S^1$.  For ease of notation, let us assume that $i=n$.  Then we have a collection $A_1,\ldots,A_{n-2}$ of vertical annuli in $\Sigma\times S^1$, with $A_i$ connecting $T_i$ to $T_{i+1}$ (c.f. the proof of Lemma \ref{lemma:mixed}).  Each annulus will have parallel, horizontal dividing curves, and we consider the neighborhood
\[
N=N(T_1\cup\cdots\cup T_{n-1}\cup A_1\cup\cdots\cup A_{n-2}),
\]
the boundary of which is $\partial N = T_1\cup\cdots\cup T_{n-1}\cup T$.  The minimal convex torus $T$ has dividing curves of slope $-(n-2)-\Sigma_{i=1}^{n-1}\lfloor\frac{q_i}{p_i}\rfloor$, and thus the toric annulus $(\Sigma\times S^1)\setminus N$ is a continued fraction block with boundary slopes $-(n-2)-\Sigma_{i=1}^{n-1}\lfloor\frac{q_i}{p_i}\rfloor$ and $\lfloor\frac{q_n}{p_n}\rfloor$.  This continued fraction block consists of
\[
\left\lfloor\frac{q_n}{p_n}\right\rfloor - \left(-(n-2)-\sum_{i=1}^{n-1}\left\lfloor\frac{q_n}{p_n}\right\rfloor\right) = (n-2) + \sum_{i=1}^n\left\lfloor\frac{q_n}{p_n}\right\rfloor = |e_0 + 2|
\]
basic slices, each of which is positive, since the stabilizations of the central knot are all positive.  In particular, we have a positive basic slice adjacent to $T_n$ whose opposite slope is $\lfloor\frac{q_n}{p_n}\rfloor-1$, measured in the coordinates of $T_n$.  We may normalize via the map
\[
\left(\begin{matrix}
\lfloor\frac{q_n}{p_n}\rfloor & -1\\
1-2\lfloor\frac{q_n}{p_n}\rfloor & 2
\end{matrix}\right)
\]
to obtain a positive basic slice with slopes $-1$ and $\infty$.  The same holds for any $1\leq i\leq n$.\\

Finally, because the $i$th leg of our Legendrian surgery diagram has a negative stabilization, we identify a negative basic slice in the solid torus $V_i$ which is adjacent to $\partial V_i$.  After gluing via $\varphi_i$, we see that $T_i=\partial V_i$ is a mixed torus.  The opposite slope $s$ of the basic slice in $V_i$ will depend as usual on the number of unstabilized knots (if any) which lie between the central knot and the first stabilized knot of the $i$th leg.  In particular, $s$ will be two more than the number of such knots.  There are then $s$ possible results of applying the JSJ decomposition along the mixed torus $T_i$, and these correspond to deleting either the central knot of our surgery diagram, the first stabilized knot of the $i$th leg, or an intermediate, unstabilized knot.  Deleting the central knot leaves us with a connected sum of lens spaces, while deleting a knot contained in the $i$th leg leaves us with a disjoint union of a lens space and a Seifert fibered space whose $i$th leg has no stabilizations.\\

Clearly the above argument may be applied to each leg with negative stabilizations (still assuming that the central knot is stabilized positively), allowing us to reduce to the case where all stabilizations in our surgery diagram have the same sign.
\end{proof}

\begin{remark}
In the classification of tight contact structures on small Seifert fibered spaces, the cases $e_0=-1$ and $e_0=-2$ are exceptional, as these are the only cases in which our space could possibly have infinitely many tight contact structures.  The case $e_0=-2$ has been studied by Ghiggini \cite{ghiggini2008tight} and Tosun \cite{tosun2020tight}, while the $e_0=-1$ case has been studied by Ghiggini-Lisca-Stipsicz \cite{ghiggini2007tight} and Matkovi{\v{c}} \cite{matkovivc2018classification}.  The existence of fillings of Sifert fibered spaces was studied in \cite{lecuona2011stein}.  As in the case $e_0\leq -3$, we may construct fillable contact manifolds with $e_0=-2$ by putting the link in Figure~\ref{fig:negative-euler} into Legendrian position.  We may then apply Proposition~\ref{prop:legs-have-single-sign} to such a diagram, to ensure that each leg has stabilizations of only one sign.  However, the lack of stabilizations on the central knot prevents us from applying Proposition~\ref{prop:single-sign}. In case $e_0=-1$, still less can be said.  In this case, Lisca-Stipsicz provide in \cite{lisca2007ozsvath} a surgery diagram which can produce all tight contact structures with maximal twisting equal to zero, but, per Etnyre-Honda \cite[Theorem 1.1 \& Lemma 3.3]{etnyre2002tight}, some such tight contact structures are not fillable.  For further analysis of these structures, see \cite{lecuona2011stein} and \cite{matkovivc2018classification}.
\end{remark}

From Propositions \ref{prop:legs-have-single-sign} and \ref{prop:single-sign} we see that the problem of classifying exact symplectic fillings for Seifert fibered spaces as in Figure \ref{fig:negative-euler} is reduced to the same problem for lens spaces, and for canonical Seifert fibered spaces.  Per above results, the lens spaces may be further reduced to universally tight lens spaces, and thus Theorem \ref{thm:negative-euler} is established.  The results of this section provide us with the usual diagrammatic calculus for reducing the classification of fillings problem: see Figure \ref{fig:reduce-to-canonical} for an example.

\subsection{Virtually overtwisted circle bundles over surfaces}\label{subsec:circle-bundle-proofs}
Honda classified the tight contact structures on circle bundles over closed Riemann surfaces in \cite[Part 2]{honda2000classification2}.  We will borrow his notation here, letting $\pi\colon M\to\Sigma$ be an oriented circle bundle over a closed, oriented surface $\Sigma$ with genus $g$.  Once we have fixed a contact structure on $M$, Honda defines the \emph{twisting number} $t(S^1)$ to be the maximum non-positive twisting number among all closed Legendrian curves in $M$ isotopic to the $S^1$-fiber, relative to the fibration framing.  The twisting number is taken to be zero if $M$ admits a fiber-isotopic Legendrian curve with positive twisting number.  We denote by $e$ the Euler number of the bundle $\pi\colon M\to\Sigma$.\\

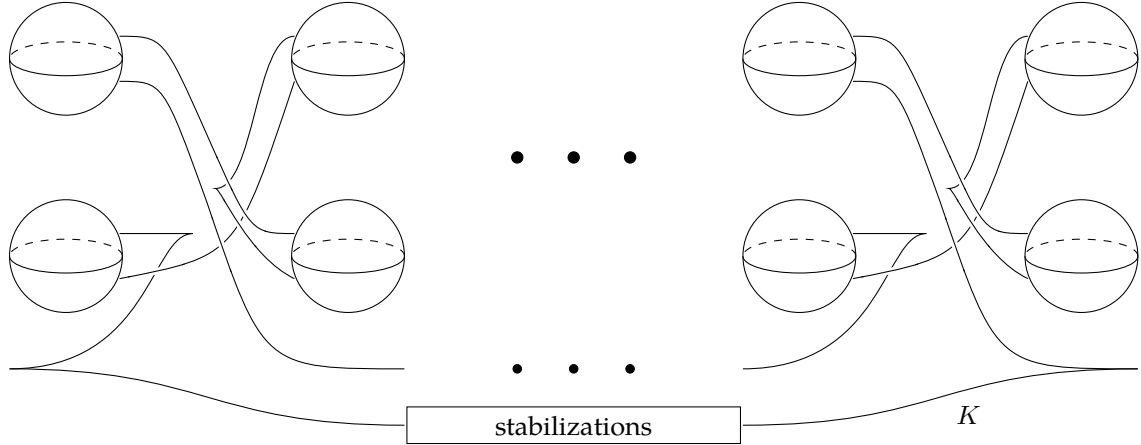
\begin{figure}
\centering
\begin{tikzpicture}[scale=0.75]
\draw (-9,3.5) circle (1);
\draw (-10,3.5) arc (180:360:1 and 0.3);
\draw[dashed] (-8,3.5) arc (0:180:1 and 0.3);
\draw (-4,3.5) circle (1);
\draw (-5,3.5) arc (180:360:1 and 0.3);
\draw[dashed] (-3,3.5) arc (0:180:1 and 0.3);
\draw (4,3.5) circle (1);
\draw (3,3.5) arc (180:360:1 and 0.3);
\draw[dashed] (5,3.5) arc (0:180:1 and 0.3);
\draw (9,3.5) circle (1);
\draw (8,3.5) arc (180:360:1 and 0.3);
\draw[dashed] (10,3.5) arc (0:180:1 and 0.3);

\draw (-9,0) circle (1);
\draw (-10,0) arc (180:360:1 and 0.3);
\draw[dashed] (-8,0) arc (0:180:1 and 0.3);
\draw (-4,0) circle (1);
\draw (-5,0) arc (180:360:1 and 0.3);
\draw[dashed] (-3,0) arc (0:180:1 and 0.3);
\draw (4,0) circle (1);
\draw (3,0) arc (180:360:1 and 0.3);
\draw[dashed] (5,0) arc (0:180:1 and 0.3);
\draw (9,0) circle (1);
\draw (8,0) arc (180:360:1 and 0.3);
\draw[dashed] (10,0) arc (0:180:1 and 0.3);

\begin{knot}[ clip width=2.5]
\strand (-8.05,3.9) .. controls (-7.275,3.9) .. (-6.5,2.15) .. controls (-5.725,0.4) .. (-4.95,0.4);
\strand (-8.05,3.1) .. controls (-7.275,3.1) .. (-6.5,0.95) .. controls (-5.5,-2) .. (-3,-2);
\strand (-4.95,-0.4) .. controls (-5.775,0) and (-6.25,1.2) .. (-6.35,1.2) .. controls (-5.65,1.2) and (-5.65,3.9) .. (-4.95,3.9);
\strand (-8.05,-0.4) .. controls (-5.98,0) .. (-4.95,3.1);
\strand (-8.05,0.4) to (-6.75,0.4) .. controls (-7.5,0.4) and(-7.5,-2) .. (-10,-2);
\strand (-10,-2) .. controls (-6.5,-2) and (-6.5,-3) .. (-3,-3);
\strand (10,-2) .. controls (6.5,-2) and (6.5,-3) .. (3,-3);
\end{knot}

\begin{scope}[xshift=13cm]
\begin{knot}[ clip width=2.5]
\strand (-8.05,3.9) .. controls (-7.275,3.9) .. (-6.5,2.15) .. controls (-5.725,0.4) .. (-4.95,0.4);
\strand (-8.05,3.1) .. controls (-7.275,3.1) .. (-6.5,0.95) .. controls (-5.5,-2) .. (-3,-2);
\strand (-4.95,-0.4) .. controls (-5.775,0) and (-6.25,1.2) .. (-6.35,1.2) .. controls (-5.65,1.2) and (-5.65,3.9) .. (-4.95,3.9);
\strand (-8.05,-0.4) .. controls (-5.98,0) .. (-4.95,3.1);
\strand (-8.05,0.4) to (-6.75,0.4) .. controls (-7.5,0.4) and(-7.5,-2) .. (-10,-2);
\end{knot}
\end{scope}

\node[draw,text width=4.2cm,align=center] at (0,-3) {stabilizations};
\filldraw[color=black, fill=black](0,1.75) circle (0.1);
\filldraw[color=black, fill=black](-1,1.75) circle (0.1);
\filldraw[color=black, fill=black](1,1.75) circle (0.1);
\filldraw[color=black, fill=black](0,-2) circle (0.075);
\filldraw[color=black, fill=black](-1,-2) circle (0.075);
\filldraw[color=black, fill=black](1,-2) circle (0.075);
\node[below] at (7,-2.5) {$K$};
\end{tikzpicture}
\caption{Stein handlebody diagrams for filling the tight contact structures on a circle bundle $\pi\colon M\to\Sigma$ with $t(S^1)=-1$.  The diagram has $2g$ 1-handles, and the knot $K$ has $2g-2-e$ stabilizations in the marked region.}
\label{fig:circle-bundle}
\end{figure}

If $2g-2>e$, Honda shows that there are $(2g-1)-e$ tight contact structures on $M$ with $t(S^1)=-1$; of these, exactly two are universally tight.  There are no virtually overtwisted contact structures on $M$ with $t(S^1)<-1$.  There are some exceptional cases of virtually overtwisted contact structures on circle bundles with $t(S^1)=0$, but these are not subject to Proposition \ref{prop:circle-bundles}.  Instead, these exceptional cases are treated by Lisca-Stipsicz \cite{lisca2004tight}.\\

Proposition \ref{prop:circle-bundles} follows immediately from Honda's description of these virtually overtwisted contact structures, as well as Theorem 1.3 of \cite{menke2018jsj}.  Namely, Honda constructs each of the $(2g-1)-e$ tight contact structures on $M$ by performing Legendrian surgery on a knot $K$ in $(\#^{2g}(S^1\times S^2),\xi_{\std})$ which has been stabilized $(2g-2)-e$ times.  Here $\xi_{\std}$ is the unique-up-to-isotopy tight contact structure on $\#^{2g}(S^1\times S^2)$.  The universally tight structures on $M$ are precisely those for which all of these stabilizations have the same sign, while each virtually overtwisted contact structure $\xi_{vot}$ results from surgery along a knot which has been stabilized both positively and negatively.  According to \cite[Theorem 1.3]{menke2018jsj}, every exact filling of $(M,\xi_{vot})$ is therefore obtained from such a filling of $(\#^{2g}(S^1\times S^2),\xi_{\std})$ by attaching a Weinstein 2-handle along $K$.  But $(\#^{2g}(S^1\times S^2),\xi_{\std})$ has a unique exact filling up to symplectomorphism, and thus the same is true of $(M,\xi_{vot})$.  This proves Proposition \ref{prop:circle-bundles}.\\

Finally, we prove Corollary~\ref{cor:circle-bundles}.

\begin{proof}[Proof of Corollary~\ref{cor:circle-bundles}]
Following the discussion in Section~\ref{subsec:circle-bundles}, we only need to verify the corollary in case $M\to \Sigma$ is a circle bundle over a torus.  In this case, we may identify our circle bundle with a parabolic torus bundle over $S^1$, and Honda's classification of tight contact structures on torus bundles \cite{honda2000classification2} tells us that a virtually overtwisted contact structure exists if and only if the Euler number $e$ of our circle bundle satisfies $|e|\geq 2$.  From \cite[Theorem 1.1(B)]{christian2021symplectic} we then know that our circle bundle admits a unique exact filling if $e\leq -2$, and admits no exact symplectic filling if $e\geq 2$.  So Corollary~\ref{cor:circle-bundles} will follow once we show that $t(S^1)=0$ in case $e\geq 2$ and $t(S^1)<0$ in case $e\leq -2$.\\

If $e\geq 2$, we simply note that, up to isotopy, the fiber of our circle bundle is a Legendrian curve $\gamma$ with $t(\gamma)=e$.  By the definition of $t(S^1)$ we have $t(S^1)=0$.\\

Next, we consider the case $e\leq -2$ and suppose that $\gamma$ is a fiber-isotopic Legendrian with $t(\gamma)=0$.  We will use a trick of Kanda \cite{kanda1997classification} to show that such a Legendrian $\gamma$ cannot exist, adapting the proof of \cite[Lemma 3.3]{honda2000classification2}.  By identifying $T^2$ as the quotient of a square $P$, we may pull back the circle bundle $M\to T^2$ to a circle bundle on $P$, giving us the solid torus $S^1\times P$.  We also pull back the tight contact structure on $M$ to a contact structure on $S^1\times P$.  In fact, by taking $P\to T^2$ to be a many-to-one quotient map --- i.e., by tiling together many copies of $P$ --- we may ensure that $\gamma$ pulls back to a closed, fiber-isotopic Legendrian curve in $S^1\times P$.  However, $S^1\times P$ is, after edge-rounding, isomorphic to a standard neighborhood of a Legendrian curve with twisting number $e\leq -2$.  But $S^1\times P$ contains a standard neighborhood of $\gamma$, which is a solid torus with slope $\infty$.  According to the classification of tight contact structures on $S^1\times D^2$ \cite{honda2000classification}, this is a contradiction.  We conclude that no such $\gamma$ exists, and thus $t(S^1)<0$.
\end{proof}

\printbibliography
\end{document}